\documentclass[a4paper,11pt]{article}

\usepackage{authblk}
\usepackage[nointlimits]{amsmath}
\usepackage{amsfonts}
\usepackage{amssymb}
\usepackage{amsmath}
\usepackage{amsthm} 
\usepackage{mathrsfs}
\usepackage{mathtools}
\usepackage{bbm}
\usepackage{cancel}
\usepackage{comment}
\usepackage{hyperref}
\usepackage[capitalise]{cleveref}
\usepackage{latexsym}
\usepackage{graphicx}
\usepackage{layout}
\usepackage[english]{babel}
\usepackage{fancyhdr}

\usepackage{lscape} 

\usepackage{amscd}
\usepackage{color}

\usepackage{tablefootnote}

\newtheorem{lemma}     {Lemma}[section]
\newtheorem{thm}   [lemma]{Theorem}
\newtheorem{teorema1}   [lemma]{Theorem}

\newtheorem{coro}    [lemma]{Corollary}
\newtheorem{cong1}      [lemma]{Conjecture}
\newtheorem{remark1}    [lemma]{Remark}
\newtheorem{defin}     [lemma]{Definition}
\newtheorem{assmp}     [lemma]{Assumption}

\newenvironment{Remark}[1][]
{\begin{remark1}[#1]\begin{samepage}}{\end{samepage}\end{remark1}}

\numberwithin{equation}{section}

\newcommand{\R}{\mathbb R}

\newcommand{\dis}{\displaystyle}

\newcommand{\mmmintone}[1]{{\dis{\int\kern -.43cm
-}}_{\kern-.21cm\substack{#1}}\;\;}
\newcommand{\mmmintwo}[2]{{\dis{\int\kern -.43cm
-}}_{\kern-.21cm\substack{#1}}^{\substack{#2}}\;\;}
\newcommand{\submint}{{\scriptstyle{\int\kern -.66em -}}}
\newcommand{\submintone}[1]{{\scriptstyle{\int\kern -.66em
-}}_{\scriptscriptstyle{\kern-.21em\substack{#1}}}}
\newcommand{\fracmint}{{\textstyle{\int\kern -.88em -}}}
\newcommand{\fracmintone}[1]{{\textstyle{\int\kern -.88em
-}}_{\scriptscriptstyle{\kern-.21em\substack{#1}}}\;}

\newcommand{\la}{\lambda}

\newcommand{\nada}[1]{}

\DeclareFontFamily{U}{BOONDOX-calo}{\skewchar\font=45 }
\DeclareFontShape{U}{BOONDOX-calo}{m}{n}{
<-> s*[1.05] BOONDOX-r-calo}{}
\DeclareFontShape{U}{BOONDOX-calo}{b}{n}{
<-> s*[1.05] BOONDOX-b-calo}{}
\DeclareMathAlphabet{\mathcalboondox}{U}{BOONDOX-calo}{m}{n}
\SetMathAlphabet{\mathcalboondox}{bold}{U}{BOONDOX-calo}{b}{n}
\DeclareMathAlphabet{\mathbcalboondox}{U}{BOONDOX-calo}{b}{n}

\usepackage{bbm}

\usepackage{setspace}
\usepackage{mathrsfs} 
\usepackage{amsmath}
\usepackage{graphicx}
\usepackage{multirow}
\usepackage{eurosym}
\usepackage{amstext}
\usepackage{hyperref}
\usepackage{grffile} 
\usepackage{color}
\usepackage{mathabx}
\usepackage{float}
\usepackage{pdfpages}
\usepackage{pifont}
\usepackage{multirow}
\usepackage[dvipsnames]{xcolor}

\usepackage[linewidth=1pt]{mdframed}
\usepackage{lipsum}

\usepackage[T1]{fontenc}
\usepackage[utf8]{inputenc} 
\usepackage{lmodern}
\usepackage{booktabs}
\usepackage{makecell}

\usepackage[normalem]{ulem}

\newcommand{\kh}[1]{{\color{teal} #1}}

\newcommand{\ch}[1]{{\color{blue} #1}}

\title{
Degree-preserving conservative processes and a unified approach for their hydrodynamics 
}

\date{}

\author[1]{Chiara Franceschini}
\author[2]{Patr\'{i}cia Gon\c{c}alves}
\author[3]{Kohei Hayashi}
\author[4]{Makiko Sasada}
\affil[1]{Dipartimento di Scienze Fisiche, Informatiche e Matematiche, Universit\`{a} degli Studi di Modena e
Reggio Emilia}
\affil[2]{Departamento de Matem\'atica, Instituto Superior T\'ecnico, Universidade de Lisboa}
\affil[3]{Department of Mathematics, Graduate School of Science, The University of Osaka}
\affil[3]{RIKEN Center for Interdisciplinary Theoretical and Mathematical Sciences~(iTHEMS)} 
\affil[4]{Graduate School of Mathematical Sciences, The University of Tokyo}
\affil[4]{RIKEN Center for Advanced Intelligence Project~(AIP)}

\begin{document}

\maketitle

\begin{abstract} 
We investigate a broad class of large-scale one-dimensional interacting systems characterized by a single conservation law and satisfying the "degree-preserving property". Under mild and natural assumptions, we establish a unified framework for the analysis of both invariant measures and hydrodynamic limits. In particular, we prove that when the generator preserves the degree of polynomials of the state variables up to order two, the marginals of any product invariant measure must belong to a family of six specific distributions. This classification is shown to be consistent with a classical result on univariate natural exponential families due to Morris~\cite{morris1982natural}, which we apply here for the first time in the context of microscopic stochastic systems. As a consequence, we construct a new interacting particle system whose invariant measure is given by the generalized hyperbolic secant distribution. 
Furthermore, we prove that, despite the generality of the dynamics, the macroscopic behavior of all models in this class is governed by the classical heat equation, with a diffusion coefficient depending explicitly on the underlying microscopic interactions. 
\end{abstract}

\section{Introduction}
A major question in statistical mechanics is to understand the large-scale behavior of the conserved quantities of microscopic random systems. In this respect, two important scaling limits are natural questions of interest. The first concerns the space-time evolution of the conserved quantities of the system. This is known in the realm of interacting particle systems as the hydrodynamic limit. The second is related to the description of the fluctuations of the random system around the typical hydrodynamical behavior of the system. Then one can also investigate non-typical profiles, which are described by large deviation principles. While many results in this context are known for specific models, developing unified approaches that apply across broad classes of systems remains a challenging task, and this work is a first step in that direction. 
\medskip

Here we focus on a class of interacting stochastic systems whose Markov generators exhibit an algebraic structure: they preserve the degree of polynomials in the state variables up to order two. We refer to such models as degree-preserving processes. In addition, we assume that the dynamics admits spatially homogeneous product invariant measures and satisfies a set of mild structural conditions (see Assumptions \ref{assumption:second_order_polymonial} and \ref{assmp:generator_element_basic}), which are met by a broad family of models of interest.
Under these assumptions, we obtain two main results. First, we show that all possible marginals of invariant product measures are highly constrained: they must belong to one of six explicit distributions. To the best of our knowledge, such a classification has not previously been established in the context of large-scale interacting stochastic systems, and it coincides with the theory of natural exponential families with quadratic variance derived by Morris in \cite{morris1982natural}.
However, note that our assertion is not just a direct application of the classical result of \cite{morris1982natural}, but it is a conceptual leap. 
Second, we provide a unified and robust derivation of the hydrodynamic limit for all models in this class, thereby placing a wide variety of systems within a common analytical framework. 
A further contribution of this work is the construction of a new model in this class of degree-preserving processes, whose invariant measure is a product of generalized hyperbolic secant distributions (GHS): we refer to it as the GHS model. 
Moreover, we also present several representative examples, illustrating each of the admissible invariant measures and their associated hydrodynamic limits. 
\medskip

To make our results precise, we begin by defining the general setting.
Throughout the present paper, we consider Markov processes with state-space $\mathscr X_N\coloneqq S^{\mathbb T_N}$ where  $S\subseteq \mathbb R$  and $\mathbb T_N=\mathbb Z/N\mathbb Z\cong \{1,2,\ldots,N\}$ denotes the one-dimensional discrete torus, and we denote by $\eta=(\eta_x)_{x\in\mathbb T_N}$ an element of the set $\mathscr X_N$,   which we refer to as a state variable.   
For simplicity, we consider the one-dimensional  one-dimensional setting with periodic boundary condition. We expect that the hydrodynamic limit results extend to models evolving on the $d$-dimensional discrete torus; however, this would require estimates for higher-dimensional random walks, and we prefer to work in one dimension in order to simplify the presentation. 
\medskip

Degree-preserving systems arise naturally and classical examples include the simple exclusion process (SEP), independent random walks (IRW), the symmetric inclusion process (SIP), as well as energy exchange models such as the Kipnis–Marchioro–Presutti (KMP) model and the Ginzburg–Landau dynamics with quadratic potential. 
See Table \ref{tab:esempio3} for the list of examples in this class.
From a microscopic viewpoint, many of these models exhibit additional algebraic structures that can be used to establish duality relations that enable explicit computations \cite{dualitybook}. In particular, models with product binomial invariant measures are associated with the $\mathfrak{su}(2)$ Lie algebra; models with product negative-binomial, Gamma, or Gaussian invariant measures are associated with the $\mathfrak{su}(1,1)$ Lie algebra; and models with product Poisson invariant measures are associated with the Heisenberg algebra. As shown in \cite{dualitybook}, different representations of these algebras give rise to different processes connected through duality. It would be interesting to investigate whether the new GHS model can also be framed within such a representation.
\medskip

From a macroscopic perspective, scaling limits for energy exchange models have been studied less extensively than those for particle systems or interacting diffusions. Our assumptions guarantee that these models belong to the class described by Macroscopic Fluctuation Theory (see Remark~\ref{MFT remark}), and it would be interesting to uncover their connections with classical integrable systems, as in \cite{Mallick2024Exact}. 
While hydrodynamic limits are well understood for many interacting particle systems, significantly less is known for energy exchange models. In the earlier work \cite{franceschini2026hydrodynamic}, hydrodynamic limits were established for several such models using attractiveness arguments. However, many degree-preserving systems remained beyond reach and  one of the contributions of the present work is to provide a robust proof of their hydrodynamic limits using mild assumptions on their Markov generators. 
\medskip

Our approach is based on a unified structural analysis with the help of the degree-preserving property.    
As in \cite{morris1982natural}, we show that the invariant product measures with quadratic variance functions are restricted to six families: normal, Poisson, gamma, binomial, negative-binomial, and generalized hyperbolic secant (GHS) distributions. 
While stochastic processes corresponding to the first five cases are known, we construct a new redistribution-type model admitting the GHS distribution as its invariant measure. Surprisingly, there is no interacting diffusion with such invariant measure, see Remark \ref{rem:GHS_not_exists}.
On the dynamical side, we prove that all degree-preserving processes satisfying our assumptions are of gradient type, i.e. their microscopic current writes as a gradient of a local function, which for the models under investigation is just a multiple of the state variable $\eta_x$. As a consequence they all share the same hydrodynamic behavior: the empirical measure converges to the solution of the classical heat equation. In other words, this means that the empirical measure associated to the state variable $\eta_x$ converges to a deterministic measure whose density is the unique weak solution of the heat equation. Since our state variables can take both positive and negative values, we consider the empirical measure process acting on test functions belonging to suitable Sobolev spaces. This universality is particularly striking given the diversity of microscopic dynamics encompassed by our framework.
\medskip

A key technical challenge is the derivation of uniform moment bounds without relying on attractiveness. Our strategy is based on the analysis of a suitable (see equation~\eqref{eq:redefine}) average of the two-point correlation function, whose evolution can be related to that of an auxiliary random walk, whose jumping rates are model dependent. By exploiting this connection and controlling the associated local times via Duhamel’s formula, we establish the necessary bounds in a robust and model-independent way. This method is flexible and applies uniformly across the class of degree-preserving processes. As a consequence, the proof of the hydrodynamic limit holds generally within such class  of models.
\medskip

Several natural directions remain open. A particularly important one is the analysis of fluctuations, i.e., establishing a central limit theorem describing deviations from the hydrodynamic limit, which is a natural and non-trivial question. We believe that our proof of moment bounds can be carried over to higher moments without many difficulties, albeit we leave this to a future work.
Another interesting question is whether processes with the degree preserving property can be fully characterized: even when the invariant distribution is fixed, there are multiple models satisfying these properties and it would be natural to determine how large the class is. 
\medskip

\begin{table}[ht]
  \centering
  \caption{List of examples of degree-preserving processes for each marginal distribution.}
  \label{tab:esempio3}
  \begin{tabular}{|l|l|l|l|l|l|l|}
    \hline
   & \makecell{Normal} & \makecell{Poisson} & \makecell{Gamma} & \makecell{Binomial} & \makecell{NBinomial} & \makecell{GHS} \\
    \hline
    \makecell{ Interacting \\ systems} & \makecell{Quadratic \\ Ginzburg- \\ Landau } & \makecell{Indep. \\ Random \\ Walkers} & \makecell{ Brownian \\ Energy} & \makecell{ Partial \\ Exclusion  } & \makecell{ Symmetric \\ Inclusion } &  \makecell{(Unknown)}  \\
    \hline
    \makecell{ Redistribution \\ systems} & \makecell{Therm.\\GL} & \makecell{Therm.\\IRW} & \makecell{Kipnis-\\Marchioro-\\ Presutti }  & \makecell{Therm.\\PEP} & \makecell{ Discrete \\KMP} & \makecell{GHS \\ Model}\\
    \hline
   \makecell{Other \\ models}  &   &  &   \makecell{ Continuous \\ Harmonic }  &  & \makecell{ Harmonic } &  \\
    \hline
  \end{tabular}
\end{table}

\subsection*{Organization of the paper}  
In \cref{sec:result}, we give the precise description of our models, assumptions and state the two main results - the first one is the classification of invariant measures into six distributions (\cref{thm:classification_main}), whereas the other one is the hydrodynamic limit (\cref{thm:hydrodynamic_general_main}). In \cref{thm:hydrodynamic_general_main}, we have imposed that under the initial measures, the average of second moments of the state variables is uniformly bounded in $N$. This is used in order to have the same bounds at any time $t$. We note that this condition is natural as it appears, for example, in the assumption of the hydrodynamic limit for the zero-range process, see \cite{kipnis1999scaling}. In \cref{sec:inv_measure_and_process}, we list examples of known models satisfying Assumptions \ref{assmp:generator_element_basic} and \ref{assumption:second_order_polymonial}, as well as a new model whose invariant measures have GHS marginals. 
The processes in the list include: interacting particles, interacting diffusions and other processes with redistribution-type interactions. 
The results of Section \ref{sec4} obtained as consequences of 
Assumptions \ref{assmp:generator_element_basic} and \ref{assumption:second_order_polymonial} together with the results of \cref{sec:proof_completion} are devoted to complete the proof of \cref{thm:hydrodynamic_general_main}. 
In subsection \ref{subsec4.1}, we show that under our assumptions all the models are of gradient form, see Lemma \ref{lem:lap}, and we also show that the quadratic variation of the martingale associated to the empirical measure is of quadratic form, see Lemma \ref{lem:qua_var}. In particular, in order to show that the limiting equation is deterministic, which boils down to showing that the aforementioned martingale vanishes as $N\to+\infty$, we need some uniform second moment bound, and the decay of the two-point correlation function.   
This is presented in subsections \ref{sebsec:main_estimates} and \ref{subsec4.2.1}.  
In \cref{sec:proof_completion}, we complete the proof of the hydrodynamic limit, namely  \cref{thm:hydrodynamic_general_main}.
Finally in Appendix \ref{sec:construction_dynamics} we prove that the models construed as redistribution-type are well-defined in terms of their Markov generators, while in Appendix \ref{sec:degree-preserving_verification} we show that the models satisfy the martingale properties of Assumption \ref{assumption:second_order_polymonial}.

\section{Models and results}
\label{sec:result}
\subsection{General setting} 
Here let us describe a class of systems that we are interested in. 
Throughout the present paper, let $\mathbb T_N=\mathbb Z/N\mathbb Z\cong \{1,\ldots, N\}$ be the discrete torus with $N$ points and whose elements are identified with modulo $N$ where $N\in\mathbb N$ is a divergent scaling parameter. 
Moreover, let $\mathbb T=\mathbb R/\mathbb Z\cong [0,1)$ be the continuous torus of length $1$. 
In what follows, let $S$ be a measurable and locally compact subset in $\mathbb R$, which in fact turns out to be $\mathbb R$, $\mathbb R_+\coloneqq [0,+\infty)$, $\mathbb N_0\coloneqq \{0,1,\ldots\}$ or $\{0,1,\ldots,\kappa\}$ with some $\kappa\in \mathbb N$, and we consider a model on the configuration space $\mathscr X_N=S^{\mathbb T_N}$. 
We denote by $\eta=\{\eta_x\}_{x\in\mathbb T_N}$ an element in $\mathscr X_N$ where $\eta_x$ stands for the state at site $x\in\mathbb T_N$. 
For each $x\in\mathbb Z$, let $\tau_x$ be the canonical shift $\tau_x\eta_z=\eta_{z+x}$ for any $z\in\mathbb T_N$.
In what follows, let $C(\mathscr X_N)$ denote the space of real-valued continuous functions on $\mathscr X_N$ when $\mathscr X_N$ is compact, and the space of real-valued continuous functions vanishing at infinity on $\mathscr X_N$ when $\mathscr X_N$ is locally compact.
The space $C(\mathscr X_N)$ is endowed with the uniform norm $  
\|f\| = \sup_{\eta \in \mathscr X_N} |f(\eta)|$,
which makes $C(\mathscr X_N)$ a Banach space.
Moreover, let us define a shift on any function on the configuration space by $\tau_x f(\eta)=f(\tau_x\eta)$. 
Now, to define a dynamics, let $L$ be an operator with the domain $\mathcal D(L)$ given by 
\begin{equation}
\label{eq:generator_definition}
L
=\sum_{x \in \mathbb{T}_N}L_{x,x+1}  
\end{equation}
where $L_{x,x+1}$ is defined by $L_{x,x+1}=\tau_x L_{0,1} \tau_{-x}$ and $L_{0,1}$ is a linear operator and assume that the domain of each operator $L_{x,x+1}$ is the same as the one of $L$.  
Throughout the paper, let $T>0$ be a fixed time horizon and assume that there is a Feller process $(\eta_x(t))_{t\ge0,x\in\mathbb T_N} $ on the space $D([0,T],\mathscr X_N)$ with Markovian generator $L$ as in \eqref{eq:generator_definition} which satisfies all items in the forthcoming assumptions (Assumptions \ref{assmp:generator_element_basic} and \ref{assumption:second_order_polymonial}). 
Above, $D([0,T],\mathscr X_N)$ denotes the space of all right-continuous trajectories with left-hand limits, taking values on the configuration space $\mathscr X_N$ and being endowed with the Skorohod topology. 
Let $\mu_N$ be a probability measure on $\mathscr X_N$ and we denote by $\mathbb P_{\mu_N}$ the probability measure on $D([0,T],\mathscr X_N)$ associated to the process $\{\eta_x(t)\}_{t\in [0,T],x\in\mathbb T_N}$, provided the distribution of $\eta(0)$ is given by $\mu_N$. 
Moreover, let $\mathbb E_{\mu_N}$ denote the expectation with respect to $\mathbb P_{\mu_N}$.  
Since the initial distribution $\mu_N$ will be fixed throughout the paper, we use the short-hand notation $\mathbb P_N=\mathbb P_{\mu_N}$ as well as $\mathbb E_N=\mathbb E_{\mu_N}$, if no confusion arises.  
Moreover, we also need to introduce the notion of the natural exponential family. A family of probability measures $\{\overline{\mu}
_{\lambda}\}_{\lambda \in I}$ on $S$ is said to be a natural exponential family (provided that $I \subset \mathbb{R}$ is not a singleton) if $\overline\mu_{\la}(ds)=(1/Z_\lambda) e^{\la s}\mu^0(ds)$ for a non-Dirac probability measure $\mu^0$ on $S$ and $I$ is the set of $\lambda \in \R$ satisfying $Z_\lambda\coloneqq \int_S e^{\lambda s}\mu^0(ds)$ is finite. Note that $I$ is a (possibly infinite) interval.

\begin{assmp}[Basic assumptions]
\label{assmp:generator_element_basic}
Assume that the operator $L_{0,1}$ satisfies: 
\begin{enumerate}
\item[\textnormal{(A1)}] Its kernel includes the constant function $1$. 

\item[\textnormal{(A2)}]
$L_{0,1} (FG)=F (L_{0,1}G)$ if $F,G, FG \in \mathcal D(L)$ and $F$ is a function of $\{\eta_y\}_{y\in\mathbb T_N\setminus \{0,1\}}$ and $\eta_0+\eta_{1}$. 

\item[\textnormal{(A3)}] $L_{0,1}F=\sigma_{0,1} (L_{0,1} \sigma_{0,1}F)$ for any $F\in \mathcal D(L)$ where $\sigma_{0,1}F(\eta)=F( \sigma_{0,1}\eta)$ and $\sigma_{0,1}\eta$ is the configuration obtained from $\eta$ by exchanging $\eta_0$ and $\eta_1$, which formally means $L_{0,1}=L_{1,0}$. 

\item[\textnormal{(A4)}] There exists a family of invariant probability measures $\{\overline{\nu}_{\la}\}_{\lambda\in I}$ for the dynamics of $\{\eta(t)\}_{t\ge 0}$, which are spatially homogeneous non-Dirac product measures whose common marginals are given by a natural exponential family. Namely, $\overline{\nu}_{\la}(d\eta)=\prod_{x \in \mathbb{T}_N}\overline{\mu}_{\la}(d\eta_x)$ where $\{\overline{\mu}
_{\lambda}\}_{\lambda \in I}$ is a natural exponential family.

\item[\textnormal{(A5)}] $E_{\overline\nu_\lambda}[F (-LF)] \ge 0$ for any $F \in \mathcal D(L)$ and for the invariant probability measure $\overline\nu_\lambda$ given in \textnormal{(A4)}.
\end{enumerate}
\end{assmp}

Next, we further introduce the stronger conditions on the model, which are crucial for deriving our two main results.
Let $\mathcal P_k$ be the space of all polynomials of $\eta_1,\ldots,\eta_N$ up to degree $k$, and we use a convention that a constant function is regarded as a degree-zero polynomial.  
We will consider the following processes for given $f\in \mathcal P_k$:
\begin{equation}
\label{eq:polynomial_martingale}
M_f(t)
\coloneqq 
f(\eta(t))
- f(\eta(0)) 
-\int_0^t Lf(\eta(s))ds
\end{equation}
and 
\begin{equation}
\label{eq:polynomial_martingale_qv}
N_f(t)\coloneqq 
M_f(t)^2 - \int_0^t
\big(Lf(\eta(s))^2 - 2 f(\eta(s))Lf(\eta(s))\big) ds,
\end{equation} 
where we impose that they are well-defined under the following assumption.

\begin{assmp}[Degree-preserving property] 
\label{assumption:second_order_polymonial}
Assume that the operator $L_{0,1}$ can act also on polynomials in $\mathcal P_2$, and it satisfies the following relations:  
\begin{itemize}
\item[\textnormal{(A6.1)}]  
$L_{0,1}\eta_0=\mathsf{p}  \eta_0+ \mathsf{q}  \eta_1 + \mathsf{r} $ for some $\mathsf{p} ,\mathsf{q} ,\mathsf{r}  \in \mathbb R$ and $L_{0,1}\eta_0 \not\equiv 0 $, and for any $f\in\mathcal P_1$, the processes $M_f(t)$ and $N_f(t)$ are martingales with respect to the natural filtration of the processes. 

\item[\textnormal{(A6.2)}] 
$L_{0,1}(\eta_0\eta_1)=\mathsf{a} (\eta_0^2+\eta_1^2)+\mathsf{b}   \eta_0\eta_1 + \mathsf{c}  (\eta_0+\eta_1)+\mathsf{d}$ for some $\mathsf{a} , \mathsf{b},\mathsf{c}  ,\mathsf{d} \in\mathbb R$ and $L_{0,1}(\eta_0\eta_1) \not\equiv  0$, and for any $f\in\mathcal P_2\setminus \mathcal P_1$, the process $M_f(t)$ is a martingale. 

\end{itemize}
\end{assmp}

It is straightforward to check that under the conditions \textnormal{(A1)}, \textnormal{(A2)} and \textnormal{(A6.1)}, we have $L_{0,1}(\sum_{x \in \mathbb{T}_N}\eta_x)=L_{0,1}(\eta_0+\eta_1)=0$ and thus $L(\sum_{x \in \mathbb{T}_N}\eta_x)=0$. 
Namely, we have at least one conservation law. 

We remark that under condition \textnormal{(A3)}, it follows automatically that $L_{0,1}(\eta_0\eta_1)$ is symmetric in $\eta_0$ and $\eta_1$. Therefore, the condition \textnormal{(A6.2)} is equivalent to requiring that $L_{0,1}(\eta_0\eta_1)$ be a polynomial in $\eta_0$ and $\eta_1$ of degree at most two.

\begin{Remark}
The condition \textnormal{(A4)} may be relaxed to the existence of a single homogeneous product invariant measure in a suitable setting. In fact, suppose that there exists an invariant probability measure $\nu^0(d\eta)=\prod_{x \in \mathbb{T}_N}\mu^0(d\eta_x)$ for the dynamics $(\eta(t))_{t \ge 0}$, which is a spatially homogeneous product measure.
Then, by assumption \textnormal{(A2)}, for 
functions $F: \R \to \R$ and $f: \mathscr X_N \to \mathbb R$, in a certain class, we have 
\begin{equation*}
L\bigg\{ F\Big(\sum_{x \in \mathbb{T}_N} \eta_x\Big) f\bigg\} 
= F\Big(\sum_{x \in \mathbb{T}_N} \eta_x\Big) Lf. 
\end{equation*}
Consequently, the expectation with respect to the measure $\nu^0$ of the last display is zero. 
Now, if $F(a)$ can be chosen to be $e^{\lambda a}$ for $\lambda$ satisfying $Z_{\lambda}< +\infty$, we obtain
$ E_{\overline{\nu}_{\lambda}}[Lf]=0$.  
Therefore, assumption \textnormal{(A4)} can in fact be weakened to the existence of a single spatially homogeneous product invariant measure under a mild condition, for instance when $S$ is finite, but we do not pursue it here. 
\end{Remark}

\begin{lemma}
\label{lem:gradient_condition}
Under Assumptions \ref{assmp:generator_element_basic} and \ref{assumption:second_order_polymonial}, $L_{0,1}\eta_0=D(\eta_1-\eta_0)$ for some $D>0$.
\end{lemma}
\begin{proof}
Conditions \textnormal{(A6.1)} and  condition \textnormal{(A3)} imply $L_{0,1}\eta_1=\mathsf{p}  \eta_1+\mathsf{q}  \eta_0 +\mathsf{r} $. Then, since $L_{0,1}(\eta_0+\eta_1)=0$ by \textnormal{(A1)} and \textnormal{(A2)}, we have $\mathsf{r} =0$ and $\mathsf{p} +\mathsf{q} =0$. 
Finally, \textnormal{(A5)} and \textnormal{(A4)} imply 
\begin{equation*}
\begin{aligned}
0 &\le E_{\overline{\nu}_{\lambda}}[\eta_0 (-L\eta_0)]=E_{\overline{\nu}_{\lambda}}[\eta_0 (-L_{0,1}\eta_0)] + E_{\overline{\nu}_{\lambda}}[\eta_0 (-L_{-1,0}\eta_0)] \\
&= E_{\overline{\nu}_{\lambda}}[\eta_0(\mathsf{p} \eta_1-\mathsf{p} \eta_0)] +E_{\overline{\nu}_{\lambda}}[\eta_0 (\mathsf{p} \eta_{-1}-\mathsf{p} \eta_0)]=-\mathsf{p} E_{\overline{\nu}_{\lambda}}[(\eta_0-\eta_1)^2].    
\end{aligned}
\end{equation*} 
Letting $D=-\mathsf{p} $, we conclude the proof. 
\end{proof}
The above lemma guarantees that all the models are of gradient type, see Lemma \ref{lem:lap}. 

\subsection{Classification of invariant measures} 
Our first main theorem is that, under Assumptions \ref{assmp:generator_element_basic} and \ref{assumption:second_order_polymonial}, we can completely characterize possible distributions of invariant measures of the dynamics.
To see that, observe that under assumption \textnormal{(A4)}, for any $\lambda \in I^{o}$, the interior of $I$, $E_{\overline{\nu}_{\lambda}}[|\eta_0|^k] <\infty$ for any $k \in \mathbb{N}$. Moreover, $\overline{\rho}: \lambda \to \overline{\rho}(\lambda)\coloneqq E_{\overline{\nu}_{\lambda}}[\eta_0]$ is strictly increasing since $\overline{\rho}'(\lambda)= \mathrm{Var}_{\overline\nu_\lambda}[\eta_0]>0$. Hence, for any $\rho \in \overline{\rho}(I^{o})$, there exists a unique $\lambda=\lambda(\rho)$ satisfying $E_{\overline\nu_{\lambda}}[\eta_0]=\rho$, in which case we use the short-hand notation $\nu_{\rho}=\overline{\nu}_{\la(\rho)}$. With this notation in place, we can state our first main theorem.

\begin{thm}
\label{thm:classification_main}
Under Assumptions \ref{assmp:generator_element_basic} and \ref{assumption:second_order_polymonial}, we have $\mathsf{a} \neq 0$, and there exist constants $v_0, v_1, v_2 \in\mathbb R$ such that $\mathrm{Var}_{\nu_\rho}[\eta_0]=v_2\rho^2+v_1\rho+v_0$ for any $\rho \in \overline{\rho} (I^{o})$, satisfying the relations 
\begin{equation*}
v_2= -\mathsf{b}/(2\mathsf{a})-1 ,\quad 
v_1=  -{\mathsf{c}}/{\mathsf{a}} \quad\text{ and }\quad 
v_0= -{\mathsf{d}}/({2\mathsf{a}}), 
\end{equation*}
for $\mathsf{a},\mathsf{b},\mathsf{c}, \mathsf{d}$ given in \textnormal{(A6.2)}.
Moreover, up to constant shifts and scaling (that is, up to an affine transformation $A\eta_0+B$ with constants $A$ and $B$), the marginal distribution of $\eta_0$ is given by either normal, Poisson, gamma, binomial, negative-binomial or the generalized hyperbolic secant distribution where the parameters $v_0,v_1$ and $v_2$ are given by 
$
(v_2,v_1,v_0)=(0,0,1), (0,1,0), (s,0,0),(-1/{\kappa},1,0), (s,1,0)$ or $(s,0,1),
$
respectively, where $s>0$ and $\kappa \in \mathbb N$. 
\end{thm}
\begin{proof}
In order to prove the first claim of the theorem, we note that $E_{\nu_{\rho}}[L(\eta_0\eta_{1})]=0$ implies that $E_{\nu_{\rho}}[L_{0,1}(\eta_0\eta_{1})]=0$. 
This follows from the fact that, using \textnormal{(A2)},
(A3) and \cref{lem:gradient_condition}, one gets that 
$E_{\nu_{\rho}}[L_{-1,0}(\eta_0\eta_{1})]=E_{\nu_{\rho}}[L_{1,2}(\eta_0\eta_{1})] = 0$.
Using condition \textnormal{(A6.2)} we have that 
\begin{equation*}
2\mathsf{a}  E_{\nu_{\rho}}[\eta_0^2] 
+\mathsf{b}   \rho^2 
+2\mathsf{c}   \rho 
+\mathsf{d}   =0 . 
\end{equation*}

If $\mathsf{a}  = 0$, then $L_{0,1} (\eta_0 \eta_1) \equiv 0 $ since $\mathsf{b}   \rho^2 +2\mathsf{c}   \rho +\mathsf{d}   =0$ holds for $\rho$ in a certain interval, which implies $\mathsf{b}  =\mathsf{c}  =\mathsf{d}   =0$. 
In \textnormal{(A6.2)} it is assumed $L_{0,1} (\eta_0 \eta_1) \not\equiv  0 $ and thus we have $\mathsf{a}  \neq 0$.
Then 
\begin{equation*}
E_{\nu_{\rho}}[\eta_0^2] 
= - \frac{\mathsf{b}  }{2\mathsf{a} } \rho^2 - \frac{\mathsf{c}  }{\mathsf{a} } \rho - \frac{\mathsf{d}    }{ 2\mathsf{a} },  
\end{equation*}
which means that
\begin{equation*}
\mathrm{Var}_{\nu_\rho}[\eta_0] 
= - \Big(1 + \frac{\mathsf{b}  }{2\mathsf{a} } \Big) \rho^2 
- \frac{\mathsf{c}  }{\mathsf{a} } \rho 
- \frac{\mathsf{d}   }{2\mathsf{a} }, 
\end{equation*}
leading to 
\begin{equation} \label{systemvs}
v_2 = -1 - \frac{\mathsf{b}  }{2\mathsf{a} }, \quad
v_1 = - \frac{\mathsf{c}  }{\mathsf{a} }, \quad
v_0 = - \frac{\mathsf{d}   }{2 \mathsf{a} } .
\end{equation}
For the second claim, in fact, in \cite[Section 4]{morris1982natural}, the author characterized all natural exponential families~(NEFs) with the quadratic variance function (see \cite[Section 2]{morris1982natural} about the terminology), namely the variance is given by a quadratic function of the mean, and it turns out that only six distributions given in the assertion are allowed up to an affine transformation.  
\end{proof}

\begin{Remark}\label{MFT remark}
In the context of applying macroscopic fluctuation theory (MFT) to the microscopic dynamics, one first computes the diffusivity $D(\rho)$ and the mobility $\sigma(\rho)$.
Under our Assumptions \ref{assmp:generator_element_basic} and \ref{assumption:second_order_polymonial}, the diffusivity and mobility of the dynamics are, respectively, a constant and a quadratic function of $\rho$. It has been noted that such a class of models can be treated in a unified manner from the MFT perspective (cf.\cite{Derrida2009Current,Mallick2024Exact}). In particular, \cite{Mallick2024Exact} pointed out that the transformation introduced in \cite{Mallick2022ExactSEP} can be applied to the MFT equations for all models with quadratic mobility, and that by relating the MFT equations to a classical integrable system via this transformation, one can obtain an exact solution of the MFT equations. That is, the examples discussed in this paper fall entirely within the scope of the method of \cite{Mallick2022ExactSEP}. For our examples, it remains an important open problem to rigorously establish the dynamical large deviation principle~(LDP) and to verify whether the results coincide with those predicted by the method of \cite{Mallick2022ExactSEP, Mallick2024Exact}. 
\end{Remark}

\subsection{Hydrodynamics}
Our interest concerning hydrodynamics is to know the macroscopic behavior of the empirical measure associated with the conserved quantity of the system, which is defined below. 
Here, given that the state on each site can also be negative, we define the empirical measure taking values in Sobolev spaces as in the setting of \cite[Section 11]{kipnis1999scaling}, see also \cite{gonccalves2025stochastic}.  
For each integer $z$, let a function $h_z$ be defined {at $u\in\mathbb T$} by 
\begin{equation}
\label{eq:sobolev_CONS_definition}
h_z(u)
= \begin{cases}
\begin{aligned}
& \sqrt 2\cos(2\pi zu) &&\text{ if } z>0, \\
& \sqrt 2\sin(2\pi zu) &&\text{ if } z<0, \\
& 1 &&\text{ if } z=0.
\end{aligned}
\end{cases}
\end{equation}
Then the set $\{h_{z},z\in\mathbb Z\}$ is an orthonormal basis of $L^{2}(\mathbb T)$. 
Consider in $L^{2}(\mathbb{T})$ the operator $1-\Delta$, which is linear, symmetric and positive. A simple computation shows that
$(1-\Delta) h_{z}=\gamma_{z}h_{z}$ where $\gamma_{z}=1+4\pi^2 z^2$.
For an integer $m\geq{0}$, denote by $H^m(\mathbb T)$ the Hilbert space induced by $C^\infty(\mathbb T)$ and the scalar product $\langle\cdot,\cdot\rangle_{m}$ defined by $\langle f,g\rangle_{m}=\langle f,(1-\Delta)^{m}g\rangle$, where $\langle\cdot,\cdot\rangle$ denotes the inner product of $L^{2}(\mathbb{T})$ and denote by $H^{-m}(\mathbb T)$ the dual of $H^m(\mathbb T)$ relatively to this inner product $\langle\cdot,\cdot\rangle_m$. 
{Note that the corresponding norm is naturally induced by this inner product.} 
Denote by $D(\mathbb R_+,H^{-m}(\mathbb T))$  the space of $H^{-m}(\mathbb T)$-valued functions, which are right-continuous and  with left limits, endowed with the Skorohod topology. 
Now, let $\pi^N_t(du)$ be  
defined by 
\begin{equation}
\label{eq:empirical_measure_definition}
\pi^N_t(du)
= \frac{1}{N}\sum_{x\in\mathbb T_N}\eta_x(N^2t) \delta_{x/N}(du)
\end{equation} 
for each $t\ge 0$, where $\delta_{x/N}$ denotes the Dirac measure with total mass on $x/N$.
Moreover, for any $G\in C(\mathbb T)$, we denote $\langle \pi^N_t,G\rangle=(1/N)\sum_{x\in\mathbb T_N}\eta_x(N^2t)G(x/N)$. 
The measure $\pi^N_t$, which is possibly a signed one, is referred to as the empirical measure. 
Hereinafter we do not specify if it is signed or not.  
Then, let us denote by $Q^{N}_m$ the measure on the Skorohod space $D([0,T],H^{-m}(\mathbb T))$ associated with the empirical measure $\pi^{N}$, which is interpret as a measure-valued function on $\mathscr X_N$, i.e., 
$
Q^{N}_m=\mathbb P_N\circ(\pi^{N})^{-1}  .
$
We show that $\pi_t^{N}$ converges, in a sense to be precised later, to a deterministic measure $\pi_t(du)$, which is absolutely continuous with respect to the Lebesgue measure, i.e., $\pi_t(du)=\rho(t,u)du$ and $\rho(t,u)$ is a solution to the heat equation.
But first we need to impose the following condition on the initial distribution of the system. 

\begin{defin}
Let $\{\mu_N\}_{N\in\mathbb N}$ be a sequence of probability measures in the state-space $\mathscr X_N$ and let {$\rho_0:\mathbb T\to\mathbb R$ be an integrable function}.
The sequence $\{\mu_N\}_{N\in\mathbb N}$ is associated to $\rho_0$ if for every $G\in C(\mathbb T)$ and for every $\delta>0$ we have that 
\begin{equation*}
\lim_{N\to\infty}\mu_N\bigg( \eta\in\mathscr X_N \,;\,  
\Big| \frac{1}{N} \sum_{x\in\mathbb T_N} \eta_x G\Big(\frac xN\Big) - \int_{\mathbb T} G(u) \rho_0(u)du \Big| > \delta\bigg) =0.
\end{equation*}
\end{defin}

Next, let us recall the notion of weak solution of the classical heat equation. 

\begin{defin}
\label{def:weak_solu}
Let $\rho_0\in L^2(\mathbb T)$. A measurable function $\rho: [0,T] \times \mathbb{T} \to \mathbb R$ is a weak solution to the heat equation with initial profile $\rho_0$
\begin{equation}  
\label{eq:HDL_heat_equation}
\begin{cases}
\begin{aligned}
& \partial_t \rho(t,u) = D\Delta \rho(t,u)
&& (t,u) \in {(0,T)} \times \mathbb{T},
\\
& \rho(0,u) = \rho_0(u) && u \in \mathbb{T}
\end{aligned}
\end{cases}
\end{equation}
if $\rho \in L^2 ([0,T]\times \mathbb T)$ and for all $t \in [0, T]$ and  $H \in C^{1,2}([0, T] \times \mathbb{T}) $ it holds
\begin{equation}\label{eq:int_sol}
\int_{\mathbb T} \rho(t,u)H(t,u) du
- \int_{\mathbb T} \rho_0(u)H(0,u) du 
- \int_0^t \int_{\mathbb T} \rho(s,u)(\partial_s + \Delta) H(s,u)du ds
= 0. 
\end{equation}
\end{defin}

Above, the space $C^{1,2}([0,T]\times \mathbb T)$ denotes the set of real-valued functions defined on $[0,T]\times \mathbb T$  that are of class $C^1$ on the first variable and $C^2$ on the second variable. 
Existence and uniqueness of weak solution of the  heat equation \eqref{eq:HDL_heat_equation} are classical and well-known. 
Indeed, we can construct a function $\rho(t,u)$ which satisfies the weak form of \eqref{eq:HDL_heat_equation} by using the heat kernel for any square integrable initial function, and it is easy to check that $\rho(t,u)$ is in $L^2([0,T]\times \mathbb T)$.  
On the other hand, uniqueness follows from an analogous argument as in \cite[Theorem A2.4.4]{kipnis1999scaling}. 
Though they assume boundedness of the initial profile, it is easy to extend the result to $L^2(\mathbb T)$ in the linear case.

Now, let us state our main theorem for the hydrodynamic limit. 

\begin{thm}
\label{thm:hydrodynamic_general_main}
Assume that the process $\{\eta(t)\}_{t\ge 0}$ satisfies all items in both Assumptions \ref{assmp:generator_element_basic} and \ref{assumption:second_order_polymonial}. 
Let $\rho_0\in L^2(\mathbb T)$ and  
let $\{\mu_N\}_{N\in\mathbb N}$ be a sequence of probability measures on $\mathscr X_N$ associated to  $\rho_0$.
Moreover, assume that  
\begin{equation}
\label{assump:initial_condition_moment} 
\sup_{N\in\mathbb N}
E_{\mu_N}\bigg[\frac{1}{N} \sum_{x\in\mathbb T_N} \eta_x^2 \bigg] < C 
\end{equation}
for some constant $C>0$. 
Then, for every $t\in [0,T]$, the sequence of empirical measures $\{\pi^{N}_t\}_{N\in\mathbb N}$ converges in probability to the absolute continuous measure $\rho(t,u)du$, that is, for any $\delta>0$ and for any $G\in C^2(\mathbb{T})$ we have 
\begin{equation*}
\lim_{N\to\infty}
Q^{N}_m 
\bigg(\Big|\langle \pi^{N}_{t},G\rangle-\int_{\mathbb T} \rho(t,u)G(u)du \Big|>\delta\bigg) =0 
\end{equation*}
where $\rho(t,u)$ is the unique weak solution to the heat equation \eqref{eq:HDL_heat_equation} provided $m>5/2$, where recall that the measure $Q^{N}_m$ was defined  for the process taking values in Sobolev space $H^{-m}(\mathbb T)$.    
\end{thm}


\begin{Remark}
Regarding the condition $L_{0,1}(\eta_0\eta_1) \not\equiv  0$ in \textnormal{(A6.2)}, we remark that if $L_{0,1}(\eta_0\eta_1) \equiv 0$, then $L (\sum_{x \in \mathbb{T}_N}\eta_x^2)=0$ holds under Assumption \ref{assmp:generator_element_basic}. 
Indeed, note that 
$$
L_{0,1} \bigg( \sum_{x \in \mathbb{T}_N}\eta_x^2\bigg)
= L_{0,1}(\eta_0^2 +\eta_1^2)=  L_{0,1}((\eta_0+\eta_1)^2 - 2\eta_0\eta_1)=0.
$$
Hence, if $L_{0,1}(\eta_0\eta_1) \equiv 0$, under only Assumption \ref{assmp:generator_element_basic}, \cref{thm:hydrodynamic_general_main} holds by the same strategy in \cref{sec:proof_completion}, whereas \cref{thm:classification_main} does not.  
In particular, if $L_{0,1}f(\eta)=f(\sigma_{0,1}\eta)-f(\eta)$, then Assumptions \ref{assmp:generator_element_basic} and \ref{assumption:second_order_polymonial} hold except for the condition $L_{0,1}(\eta_0\eta_1)\not\equiv  0$, and for any $S$, any spatially homogeneous product measures on $S^{\mathbb{T}_N}$ are invariant for the dynamics.
\end{Remark}

\section{Examples of interacting systems satisfying our assumptions} 
\label{sec:inv_measure_and_process}
Recall from \cref{thm:classification_main} that the variance of the state variable with respect to the invariant measure is given as a degree-two polynomial of the parameter $\rho$.    
Moreover, again by Theorem \ref{thm:classification_main} only six distributions having quadratic variance functions - normal, Poisson, gamma, binomial, negative binomial~(NB), generalized hyperbolic secant~(GHS) distributions, are allowed as invariant measures of each process.  
To denote one of these six distributions, we use the notation $\sigma\in \{ \mathrm{Normal}, \mathrm{Poisson}, \mathrm{Gamma}, \mathrm{Binomial}, \mathrm{NB}, \mathrm{GHS} \}$ and let $V^\sigma$ be the corresponding variance function, which takes the form of
\begin{equation*}
V^\sigma(\rho)=v_0+v_1\rho + v_2\rho^2
\end{equation*}
for some $v_0$, $v_1$ and $v_2$. 
In this part, we give examples of large-scale interacting systems which admit product invariant measure whose common marginal is one of the aforementioned six distributions, and satisfy the assumptions detailed above.  
Here we check the degree-preserving properties in assumptions \textnormal{(A6.1)} and \textnormal{(A6.2)} whereas the others can easily be verified. 
In particular, for \textnormal{(A4)} one can check the stationary condition or the detailed balance condition \eqref{eq:redistribution_detailed_balance} given below, while for the non-negativity of the Dirichlet form \textnormal{(A5)}, see ~\cite[A1.10 of Section]{kipnis1999scaling}.  
We show that redistribution-type interactions always define a process with the prerequisites, however this is not the case for interacting systems as there is no known corresponding model for the GHS distribution, see \cref{rem:GHS_not_exists}.  
A redistribution process on the state-space $\mathscr X_N=S^{\mathbb T_N}$ is generated by the following operator: 
\begin{equation}
\label{eq:generator_redistribution}
L^{\mathrm{Red}}
=
{\frac{1}{2} \sum_{x,y \in\mathbb T_N, |x-y|=1} 
\nabla_{x,y} }
\end{equation}
where $\nabla_{x,y}$ is defined by 
\begin{equation}\label{eq:nabla}
\nabla_{x,y} f(\eta)
= \int_{S} q_\alpha(\eta_x,\eta_y)
[f(\eta-\alpha \delta_x + \alpha \delta_y) - f(\eta )] d\alpha 
\end{equation}
for any $f\in C(\mathscr X_N)$ with some non-negative rate $q_\alpha:S\times S\to\mathbb R$ that we describe below. 
Above, $\delta_x$ denotes the configuration with unitary value at site $x\in \mathbb T_N$ and zero elsewhere.
Additionally, $d\alpha$ denotes the Lebesgue measure if $S=\mathbb R$ or $\mathbb R_+$, whereas if $S\subset \mathbb Z$ the integration in \eqref{eq:nabla} is read as summation .
Let $p^\sigma$ be the probability density function of one of the six distributions: $\sigma\in \{ \mathrm{Normal}, \mathrm{Poisson}, \mathrm{Gamma}, \mathrm{Binomial}, \mathrm{NB}, \mathrm{GHS} \}$. 
Now, let us take the rate $q_\alpha=q_\alpha^\sigma$ in \eqref{eq:generator_redistribution} as  
\begin{equation}
\label{eq:redistribution_jump_rate_without_indicator} 
\begin{aligned}
q^\sigma_\alpha(\eta_x,\eta_y)
&=c^\sigma_\alpha(\eta_x\eta_y) \mathbf{1}_{\{{\eta_x,\eta_y,}\eta_x-\alpha,\eta_y+\alpha \ \in \ S\}} =\frac{
p^\sigma(\eta_x-\alpha)  p^\sigma(\eta_y+\alpha)}{Z^\sigma(\eta_x+\eta_y) }
\mathbf{1}_{\{{\eta_x,\eta_y,}\eta_x-\alpha,\eta_y+\alpha \ \in \ S\}} 
\end{aligned}
\end{equation} 
with 
\begin{equation}
\label{eq:partition_function_redistribution_def}
Z^\sigma(\eta_x+\eta_y)
= \int_S p^\sigma(\eta_x-\alpha')  p^\sigma(\eta_y+\alpha')
\mathbf{1}_{\{{\eta_x,\eta_y,} \eta_x-\alpha',\eta_y+\alpha'\in S\} }
d\alpha',
\end{equation}
where we used the fact that the normalizing factor $Z^\sigma$ depends only on $\eta_x+\eta_y$, which is verified by a transformation $\alpha'\mapsto \eta_x-\alpha'$. 
In other words, the rate is normalized in such a way that 
\begin{equation}
\label{eq:redistribution_rate_normalization}
\int_S q^\sigma_\alpha(\eta_x,\eta_y) d\alpha=1
\end{equation}
for any $\eta_x,\eta_y\in S$. 
This means that the redistribution rate $q^\sigma_\cdot(\eta_x,\eta_y)$ defines another probability distribution, which becomes normal, binomial, beta, hypergeometric, negative-hypergeometric or GHS, provided the marginal is given by normal, Poisson, gamma, binomial, NB or GHS, respectively. 

Here, following a basic recipe to construct a Markov process from an operator, see \cite[Section 3.3]{liggett2010continuous}, we can show the existence of a Feller process $\{ \eta(t):t\ge 0\}$ on $\mathscr X_N$ with generator $L^{\mathrm{Red}}$.
For readers' convenience, we explain some details on this construction of the dynamics in \cref{sec:construction_dynamics}.

Additionally, we note that the process generated by $L^{\mathrm{Red}}$ admits the product invariant measure whose common marginal is given by one of the six distributions, due to the detailed balance condition: 
\begin{equation}
\label{eq:redistribution_detailed_balance}
q^\sigma_\alpha(\eta_x,\eta_y) 
p^\sigma(\eta_x) p^\sigma(\eta_y) 
= q^\sigma_\alpha(\eta_y + \alpha,\eta_x - \alpha) p^\sigma(\eta_x -\alpha) p^\sigma(\eta_y + \alpha). 
\end{equation}
Moreover, it is easy to see that $c_\alpha^\sigma$ and $q^\sigma_\alpha$ are invariant under the transformation $\alpha\mapsto \eta_x-\eta_y-\alpha$. 
Therefore, we have that 
\begin{equation}
\label{eq:redistribution_degree-one_preservation}
\begin{aligned}
\int_S \alpha q^\sigma_\alpha(\eta_x,\eta_y)
d\alpha
&=\int_S(\eta_x-\eta_y-\alpha) q^\sigma_\alpha(\eta_x,\eta_y)
d\alpha
=\frac{1}{2}(\eta_x-\eta_y)
\end{aligned}
\end{equation}
where we used \eqref{eq:redistribution_rate_normalization}, and thus the preservation of degree-one terms is straightforward.
This particularly yields $L\eta_x=\Delta \eta_x$ for all cases. 
Moreover, for the second-order polynomial, it is enough to compute the second moment with respect to the probability measure $q^\sigma_\alpha(\eta_x,\eta_y) d\alpha$ and see that it takes a quadratic form: 
\begin{equation}
\label{eq:adjoint_probability_second_moment_quadratic}
\begin{aligned}
M_2^\sigma(\eta_x,\eta_y)
&\coloneqq \int_S \alpha^2 q^\sigma_\alpha(\eta_x,\eta_y) d\alpha\\
&= \Gamma^\sigma_{2,0}\eta_x^2 + \Gamma^\sigma_{1,1}\eta_x\eta_y
+ \Gamma^\sigma_{0,2}\eta_y^2 
+ \Gamma^\sigma_{1,0}\eta_x + \Gamma^\sigma_{0,1}\eta_y
+\Gamma^\sigma_{0,0}.     
\end{aligned}
\end{equation}
Indeed, noting $L^{\textrm{Red}}_{x,x+1}=\nabla_{x,x+1}+\nabla_{x+1,x}$ and  
\begin{equation*}
\begin{aligned}
L^{\textrm{Red}}_{0,1}(\eta_0\eta_1)
&=\int_S 
{q^\sigma_\alpha(\eta_0,\eta_1)} [(\eta_0-\alpha)(\eta_1+\alpha)-\eta_0\eta_1]d\alpha 
\\&\quad
+ \int_S {q^\sigma_\alpha(\eta_1,\eta_0)} [(\eta_0+\alpha)(\eta_1-\alpha)-\eta_0\eta_1]d\alpha \\
&=(\eta_0-\eta_1)^2 
-M_2(\eta_0,\eta_1) - M_2(\eta_1,\eta_0)\\
&= (1-\Gamma^\sigma_{2,0}-\Gamma^\sigma_{0,2})(\eta_0^2+\eta_1^2) 
{- (2+2\Gamma^\sigma_{1,1})\eta_0\eta_1}
- (\Gamma^\sigma_{1,0}+\Gamma^\sigma_{0,1})(\eta_0+\eta_1)
-2\Gamma^\sigma_{0,0}, 
\end{aligned}
\end{equation*}
which satisfies \textnormal{(A6.2)}. 
Hence, it remains to show for each of the six distributions that \eqref{eq:adjoint_probability_second_moment_quadratic} holds. 

\begin{Remark} 
Note that each redistribution model corresponds to the thermalized version of its associated interacting system, whenever such version exists. It would be interesting to investigate whether the Harmonic model and the GHS model can also be obtained as thermalization limit of appropriate interacting systems.  We leave this to a future work.  
We also remark that interacting systems of discrete occupation variables, as for example  independent random walkers, the partial exclusion process and the inclusion process are special cases of misanthrope processes with decomposable rates. On the other hand the redistribution-type models of discrete occupation variables are mass migration processes, since more than one particle per site is allowed to jump \cite{fajfrova2016invariant}. Finally, the Harmonic model is also a mass migration process with special decomposable rate which only depend on the occupation number of the departure site, i.e. of zero-range type. 
\end{Remark}

\subsection{Normal distribution}
Take $S=\mathbb R$. Let $\nu_{\rho}^{\mathrm{Normal}}$ be the product Gaussian measure whose common marginal is given by the normal distribution: 
\begin{equation*}
\nu^{\mathrm{Normal}}_\rho(\eta_0\le z)
=\int_{-\infty}^z
p^{\mathrm{Normal}}(y)dy,\quad
p^{\mathrm{Normal}}(y)
=  (2\pi \sigma^2)^{-1/2}e^{-(y-\rho)^2/(2\sigma^2)} .
\end{equation*}
In this case, the variance function $V^{\mathrm{Normal}}(\rho)=\sigma^2$ becomes constant.

\subsubsection{Redistribution}
Here let us consider a process on $\mathscr X_N$ with infinitesimal generator given as in \eqref{eq:generator_redistribution} with $\nabla_{x,y}$ which is defined by  
\begin{equation} \label{redNormal}
\nabla_{x,y}f(\eta)
=\int_{-\infty}^\infty q^{\mathrm{Normal}}_\alpha(\eta_x,\eta_y) 
[f(\eta-\alpha\delta_x+\alpha \delta_y) - f(\eta)]d\alpha 
\end{equation}
for any $f:\mathscr X_N\to\mathbb R$ and in this case 
we have 
\begin{equation*}
\begin{aligned}
Z^{\mathrm{Normal}}(\eta_x+\eta_y)
&=\int_{-\infty}^\infty
p^{\mathrm{Normal}}(\eta_x-\alpha)
p^{\mathrm{Normal}}(\eta_y+\alpha) d\alpha\\
&= (4\pi \sigma^2)^{-1/2} 
\exp(-\tfrac{1}{4\sigma^2} (\eta_x+\eta_y-2\rho )^2 )
\end{aligned}
\end{equation*}
and the rate satisfies the normalization \eqref{eq:redistribution_rate_normalization}, and thus preservation of degree-one term is deduced from \eqref{eq:redistribution_degree-one_preservation}. Moreover,
the non-negative rate reads 
\begin{align*}
c^{\mathrm{Normal}}_\alpha(\eta_x,\eta_y)
&= \frac{1}{\sqrt{\pi\sigma^2}} 
e^{-\tfrac{1}{2\sigma^2}(\eta_x-\alpha-\rho)^2} 
e^{-\tfrac{1}{2\sigma^2}(\eta_y+\alpha-\rho)^2} 
e^{\tfrac{1}{4\sigma^2}(\eta_x+\eta_y-2\rho)^2} 
= \frac{1}{\sqrt{\pi\sigma^2}}e^{-\tfrac{1}{\sigma^2}\left(\alpha-\tfrac{\eta_x - \eta_y}{2}\right)^2}   
\end{align*}
Then, by definition, the process is reversible with respect to the measure $\nu_\rho^{\mathrm{Normal}}$ by the detailed balance condition \eqref{eq:redistribution_detailed_balance}. 
Now, let us check that \eqref{eq:adjoint_probability_second_moment_quadratic} holds, so that degree-two terms are preserved under the dynamics. 
This is straightforward since {it is the second  moment of a random variable distributed as $\mathcal{N}(\tfrac{\eta_x-\eta_y}{2}, \tfrac{\sigma^2}{2})$}
\begin{equation*}
\begin{aligned}
M_2^{\mathrm{Normal}}(\eta_x,\eta_y)
=\frac{1}{\sqrt{\pi \sigma^2}} \int_{-\infty}^\infty \alpha^2 \exp(-\tfrac{1}{\sigma^2}(\alpha-\tfrac{\eta_x-\eta_y}{2})^2) d\alpha 
= \frac{\sigma^2}{2} + \frac{1}{4}(\eta_x-\eta_y)^2 \;,
\end{aligned}
\end{equation*}
which is a quadratic polynomial in $\eta_x$ and $\eta_y$.

\subsubsection{Interacting diffusion} 
Here, let us give other examples for which the product Gaussian measure $\nu_\rho^{\mathrm{Normal}}$ is invariant under the dynamics. 
Typically, the normal distribution arises as an invariant measure of interacting Brownian motions, which is described by a system of stochastic differential equations. 
Let us firstly consider the following Ginzburg-Landau model whose infinitesimal generator is given by 
\begin{equation*}
L^{\mathrm{GL}}f(\eta)
= \frac{1}{4}\sum_{x,y\in\mathbb T_N,\, |x-y|=1}(\partial_{y}-\partial_x)^2 f(\eta)
-\frac{1}{4\sigma^2}\sum_{x,y\in\mathbb T_N,\, |x-y|=1}(\eta_{y}-\eta_x)(\partial_{y}-\partial_x)f(\eta) ,
\end{equation*}
for any smooth function $f:\mathscr X_N\to\mathbb R$, where $\partial_x=\partial/\partial\eta_x$ denotes the derivative with respect to $\eta_x$. 
Letting $\partial_x^*$ be the adjoint operator of $\partial_x$ with respect to the measure $\nu_{\rho}^{\mathrm{Normal}}$, a simple computation shows $\partial_x^*=-\partial_x +(\eta_x-\rho)/\sigma^2$. 
Then, we have that $L^{\mathrm{GL},*}=L^{\mathrm{GL}}$, which particularly means that the measure $\nu_\rho^{\mathrm{Normal}}$ is invariant under the dynamics of the Ginzburg-Landau model. 
One might wonder what is the thermalization limit of the Ginzburg-Landau diffusion. Considering the bond $(x,y)$, its reversible measure is $p^{\mathrm{Normal}}(\eta_x) p^{\mathrm{Normal}}(\eta_y)$ and computing the density of $\eta_x$ conditioning on having the total bond momenta $\eta_x + \eta_{y}$ fixed, one gets exactly the generator \eqref{redNormal} after performing a suitable change of variable $\alpha = u - \eta_y$.

\subsection{Poisson distribution} 
Next, we consider the case where the variance is a linear function of the density. 
This is only possible for the Poisson distribution, taking $S=\mathbb N_0$.  
Let $\nu_\rho^{\mathrm{Poisson}}$ be the product Poisson distribution given by 
\begin{equation*}
\nu_\rho^{\mathrm{Poisson}}(\eta_0=k)
= p^{\mathrm{Poisson}}(k)
=\frac{1}{k!}\rho^k e^{-\rho},
\end{equation*}
for each $k\in S$.  
Then the variance function is given as $V^{\mathrm{Poisson}}(\rho)=\rho$.

\subsubsection{Redistribution} 
In this case, the process  is generated by \eqref{eq:generator_redistribution} where $\nabla_{x,y}$ is defined by 
\begin{equation*}
\nabla_{x,y}f(\eta)
= \sum_{\alpha =-\eta_y}^{\eta_x} c^{\mathrm{Poisson}}_\alpha(\eta_x,\eta_y) 
[f(\eta-\alpha\delta_x + \alpha \delta_y) - f(\eta)] 
\end{equation*}

for any $f:\mathscr X_N\to\mathbb R$, and  from an elementary identity we get 
\begin{equation*}
Z^{\mathrm{Poisson}}(\eta_x+\eta_y)
=\sum_{\alpha=-\eta_y}^{\eta_x} 
p^{\mathrm{Poisson}} (\eta_x-\alpha) p^{\mathrm{Poisson}} (\eta_y+\alpha) 
= \frac{1}{(\eta_x+\eta_y)!} (2\rho)^{\eta_x+\eta_y} e^{-2\rho}.
\end{equation*} Moreover, the non-negative rate is given by 
\begin{equation*}
c^{\mathrm{Poisson}}_\alpha(\eta_x,\eta_y) 
= \frac{1}{2^{\eta_x+\eta_y}} 
\binom{\eta_x+\eta_y}{\eta_x-\alpha}
= \frac{1}{2^{\eta_x+\eta_y}} 
\binom{\eta_x+\eta_y}{\eta_y+\alpha}. 
\end{equation*}
Therefore, the rate $c_\alpha^{\mathrm{Poisson}}$ is normalized as \eqref{eq:redistribution_rate_normalization}, and thus we know from \eqref{eq:redistribution_degree-one_preservation} that the process preserves degree-one occupation variables. 
Additionally, to check \eqref{eq:adjoint_probability_second_moment_quadratic}, 
\begin{equation*}
\begin{aligned}
M_2^{\mathrm{Poisson}}(\eta_x,\eta_y)
&= \frac{1}{2^{\eta_x+\eta_y}} 
\sum_{\beta=0}^{\eta_x+\eta_y} 
(\eta_x-\beta)^2 
\binom{\eta_x+\eta_y}{\beta}\\
&= \frac{1}{2^{\eta_x+\eta_y}} 
\sum_{\beta=0}^{\eta_x+\eta_y} 
\big[\beta(\beta-1) 
-(2\eta_x-1)\beta 
+ \eta_x^2\big] 
\binom{\eta_x+\eta_y}{\beta}\\
&= \frac{1}{4}(\eta_x+\eta_y)(\eta_x+\eta_y-1) 
- \frac{1}{2}(\eta_x+\eta_y) (2\eta_x-1) 
+ \eta_x^2 \\
&= \frac{1}{4} (\eta_x^2-2\eta_x\eta_y+\eta_y^2 + \eta_x+\eta_y),
\end{aligned}
\end{equation*}
which yields the preservation of degree-two terms as well. 
Let us comment here that this model was originally introduced in \cite[Section 5]{carinci2013duality} as a model arising from independent random walkers after an instantaneous thermalization limit.

\subsubsection{Interacting particles} 
As another example of associated processes, let us consider the zero-range process with jump rate equal to the occupation variable $\eta_x$, which, in fact, turns out to be independent random walks. 
See \cite[Section 2]{kipnis1999scaling} for more expositions on zero-range processes.   
The infinitesimal generator of the process is given by 
\begin{equation*}
L^{\mathrm{IRW}}f(\eta) 
=\frac{1}{2}  \sum_{x,y\in\mathbb T_N,{|x-y|=1}} \eta_x [f(\eta^{x,y})- f(\eta)]
\end{equation*}
for any $f:\mathscr X_N\to\mathbb R$ where $\eta^{x,y}$ denotes a configuration where a particle jumps from $x$ to $y$:
\begin{equation*}
(\eta^{x,y})_z
= \begin{cases}
\begin{aligned}
& \eta_x-1 && \text{ if } z=x,\\
& \eta_y+1 && \text{ if } z=y,\\
& \eta_z && \text{ otherwise. }
\end{aligned}
\end{cases}
\end{equation*}
Then, it is not hard to check that $E_{\nu_\rho^{\mathrm{Poisson}}}[L^{\mathrm{IRW}}f(\eta)]=0$ for any $f:\mathscr X_N\to\mathbb R$, and thus the process of independent random walks is reversible under the measure $\nu_\rho^{\mathrm{Poisson}}$, and the process preserves the degree up to two.

\subsection{Gamma distribution}
Let $\nu_\rho^{\mathrm{Gamma}}$ be a product measure whose common marginal is given by a Gamma distribution with shape parameter $2\mathfrak s > 0$ and free scale parameter $\rho/(2\mathfrak s) > 0$: 
\begin{equation*}
\nu_{\rho}^{\mathrm{Gamma}}(\eta_0\le z) 
= \int_0^\infty p^{\mathrm{Gamma}}(y)dy
= \frac{1}{\Gamma(2\mathfrak s) (\rho/(2\mathfrak s))^{2\mathfrak s}} \int_0^z y^{2\mathfrak s-1} e^{-2\mathfrak sy / \rho } dy 
\end{equation*}
for any $z\ge 0$, where $\Gamma(z)=\int_0^\infty t^{z-1}e^{-t}dt$ is the Gamma function. 
Note that the measure is parametrized by density, i.e., $E_{\nu^{\mathrm{Gamma}}_\rho}[\eta_0]=\rho$ and the variance function is purely quadratic: $V^{\mathrm{Gamma}}(\rho)=\rho^2/(2\mathfrak s)$.

\subsubsection{Redistribution} 
Now, let us consider a Markov process on the configuration space $\mathscr X_N$ with redistribution-type interaction. 
The corresponding process is referred to as the generalized Kipnis-Marchioro-Presutti~(KMP) model, which was introduced in \cite{giardina2009duality} as a generalization of the original energy transport model \cite{kipnis1982heat}.  
The readers can find more expositions on this process in the previous work \cite[Section 2.1.1]{franceschini2026hydrodynamic}, but the parametrization of the measure is slightly different.  
In the definition of the generator \eqref{eq:generator_redistribution}, the operator $\nabla_{x,y}$ is defined for any $f:\mathscr X_N\to\mathbb R$ by  
\begin{equation*}
\nabla_{x,y}f(\eta)
= \int_{-\eta_y}^{\eta_x}  
c_\alpha^{\mathrm{Gamma}}(\eta_x,\eta_y)
[f(\eta-\alpha \delta_x+\alpha\delta_y) - f(\eta)]d\alpha.
\end{equation*}
From an elementary computation regarding the Beta distribution, we note that  
\begin{equation*}
Z^{\mathrm{Gamma}}(\eta_x+\eta_y)
=\int_{-\eta_y}^{\eta_x}
(\eta_x-\alpha)^{2\mathfrak s-1} 
(\eta_y+\alpha)^{2\mathfrak s-1} d\alpha
= (\eta_x+\eta_y)^{4\mathfrak s-1}\frac{\Gamma(2\mathfrak s)^2}{ \Gamma(4 \mathfrak s)},
\end{equation*} and then
the rate reads  as 
\begin{equation*}
c_\alpha^{\mathrm{Gamma}}(\eta_x,\eta_y)
= \frac{\Gamma(4\mathfrak s)}{\Gamma(2\mathfrak s)^2}    \dfrac{(\eta_x - \alpha)^{2\mathfrak s-1}(\eta_{y} + \alpha)^{2\mathfrak s-1}}{(\eta_x + \eta_{y})^{4\mathfrak s-1}} \;.
\end{equation*}
Note that the classical parametrization, as in \cite{kipnis1982heat}, is obtained under the change of variables $\alpha = \eta_x - u(\eta_x + \eta_y)$, with $u \in [0,1]$. 
By definition, the process is reversible with respect to the measure $\nu_\rho^{\mathrm{Gamma}}$.  
From the identities above it can be easily checked that the normalizing condition \eqref{eq:redistribution_rate_normalization} is satisfied.  
Therefore, as before, we know from \eqref{eq:redistribution_degree-one_preservation} that degree-one terms are preserved under the action of the generator. 
Moreover, for \eqref{eq:adjoint_probability_second_moment_quadratic}, note that 
\begin{equation*}
\begin{aligned}
M_2^{\mathrm{Gamma}}(\eta_x,\eta_y)
&= \int_{-\eta_y}^{\eta_x}
\alpha^2 c_\alpha^{\mathrm{Gamma}}(\eta_x,\eta_y)d\alpha \\
&= \frac{\Gamma(4\mathfrak s)}{\Gamma(2\mathfrak s)^2} 
(\eta_x+\eta_y)^2 
\int_0^1 \Big( \beta -\frac{\eta_x}{\eta_x+\eta_y}\Big)^2  
\beta^{2\mathfrak s-1} (1-\beta)^{2\mathfrak s-1} d\beta\\
&= \frac{\Gamma(4\mathfrak s)}{\Gamma(2\mathfrak s)^2} 
(\eta_x+\eta_y)^2 
\int_0^1 \Big( \beta-\frac{1}{2} 
- \frac{\eta_x-\eta_y}{2(\eta_x+\eta_y)} \Big)^2  
\beta^{2\mathfrak s-1} (1-\beta)^{2\mathfrak s-1} d\beta\\
&= \frac{(\eta_x+\eta_y)^2}{4(4\mathfrak s+1)} 
+ \frac{(\eta_x-\eta_y)^2}{4}
= \frac{2\mathfrak s+1}{2(4\mathfrak s+1)}  
(\eta_x^2+\eta_y^2) 
- \frac{2\mathfrak s}{4\mathfrak s+1} \eta_x\eta_y  
\end{aligned}
\end{equation*}
where we used the fact that the mean (resp. variance) of the Beta distribution with parameter $(2\mathfrak s,2\mathfrak s)$ is given by $1/2$ (resp. $1/[4(4\mathfrak s+1)]$).
Hence the generalized KMP model preserves degree-two terms.

\subsubsection{Interacting diffusion} 
Here, let us give an example {known as the} Brownian energy process~(BEP) in the class of interacting diffusions. 
The Brownian energy process was introduced in \cite{giardina2009duality} by considering an energy (i.e., the square of the variables) of the Brownian momentum process. 
The generator of the process is defined by 
\begin{equation*}
L^{\mathrm{BEP}}f(\eta) 
=\sum_{x,y\in\mathbb T_N,\,  |x-y|=1}  
\big( 2\eta_x\eta_{y}(\partial_{\eta_x}-\partial_{\eta_{y}})^2 
-\mathfrak{s} (\eta_x-\eta_{y})(\partial_{\eta_x}-\partial_{\eta_{y}} \big)
\big) f(\eta), 
\end{equation*}
which acts on smooth functions $f$ on $\mathscr X_N$ and every $\eta\in \mathscr X_N$.  
Then by \cite[Theorem 6.1]{giardina2009duality}, an invariant measure of the process is given by the product measure whose common marginal is the chi-square distribution $\chi_{\mathfrak s}^2$ with $\mathfrak s$-degree of freedom, and it is a special case of the gamma distribution with shape parameter $\mathfrak s/2$ and scale parameter $2$. 
Moreover, note that it is clear from the form of the generator that the  Brownian energy process preserves the degree of polynomials with arbitrary order. 

\subsubsection{Continuous harmonic model}
This family of energy models arise as the Markov duals of the Harmonic models presented in subsection \ref{harmonicmodel}, see \cite{franceschini2023integrable}.
We consider the dynamics on the one-dimensional discrete torus, which is conservative, noting that Proposition 2.4 of \cite{kim2025spectral} establishes that the model is well-defined on any finite graph.

The generator of the continuous harmonic models labeled by the spin parameter $\mathfrak s>0$ is given by \eqref{eq:generator_redistribution} with $\nabla_{x,y}=\nabla^{\mathrm{CHarm}}_{x,y}$ where $\nabla^{\mathrm{CHarm}}_{x,y}$ is defined for each $f:\mathscr X_N\to \mathbb R$ by  
\begin{equation*}
\nabla^{\mathrm{CHarm}}_{x,y} f(\eta) 
= \int_0^{\eta_x} \dfrac{1}{\alpha} \left( 1- \frac{\alpha }{\eta_x}\right)^{2\mathfrak{s} -1} \left[ f(\eta - \alpha \delta_x + \alpha \delta_y) - f(\eta) \right] \;.
\end{equation*}
In other words, an energy $\alpha > 0$ moves from site $x$ to site $y$ at rate 
\begin{equation}
c^{\mathrm{CHarm}}_\alpha (\eta_x,\eta_y)
= \dfrac{1}{\alpha} \left( 1- \frac{\alpha }{\eta_x}\right)^{2\mathfrak{s} -1}  \mathbf{1}_{\lbrace \alpha \leq \eta_x \rbrace} \;.
\end{equation}
The duality relation with the harmonic model of subsection \ref{harmonicmodel} was proved in \cite{franceschini2023integrable} and as a consequence allows to infer stationary (see Theorem 2.1 which can easily be adapted with our parametrization) by classifying all the moments. Moreover, having a product duality function of moments-type with a dual conservative interacting particle systems guarantees that the generator conserves any degree of the polynomials of state variables.

\subsection{Binomial distribution}
Let us take $S=\{0,1,\ldots,\kappa\}$ with $\kappa\in\mathbb N$, and let $\nu_\rho^{\mathrm{Binomial}}$ be the product measure whose common marginal is given by the binomial distribution with mean $\rho$:
\begin{equation*}
\nu_\rho^{\mathrm{Binomial}}(\eta_0=m) 
=p^{\mathrm{Binomial}}(m)=\binom{\kappa}{m} 
(\rho/\kappa)^m(1-\rho/\kappa)^{\kappa-m}, \quad {\textrm{for} \quad m=0,\cdots, \kappa}. \end{equation*}
Note that the variance function for this case is given by $V^{\mathrm{Binomial}}(\rho)=-\rho^2/\kappa + \rho$. 

\subsubsection{Redistribution}

In this case, the generator \eqref{eq:generator_redistribution} is acting on $f:\mathscr X_N\to\mathbb R$ as
\begin{equation*}
\nabla_{x,y}f(\eta)
= \sum_{\alpha=\max\{\eta_x-\kappa,-\eta_y\} }^{\min\{\eta_x, \kappa-\eta_y\} } 
c^{\mathrm{Binomial}}_\alpha(\eta_x,\eta_y) 
[f(\eta-\alpha\delta_x+\alpha\delta_y) - f(\eta)].
\end{equation*}
Note that by a change of variable $\beta=\eta_x-\alpha$, we know that  
\begin{equation*}
\begin{aligned}
Z^{\mathrm{Binomial}}(\eta_x+\eta_y)
&=(\rho/\kappa)^{\eta_x+\eta_y}(1-\rho/\kappa)^{2\kappa-(\eta_x+\eta_y)}\sum_{\alpha=\max\{\eta_x-\kappa,-\eta_y\} }^{\min\{\eta_x, \kappa-\eta_y\} } 
\binom{\kappa}{\eta_x-\alpha}
\binom{\kappa}{\eta_y+\alpha}\\
&= (\rho/\kappa)^{\eta_x+\eta_y}(1-\rho/\kappa)^{2\kappa-(\eta_x+\eta_y)}\sum_{\beta=\max\{0,\eta_x+\eta_y-\kappa\}}^{ \min\{\eta_x+\eta_y,\kappa \} } 
\binom{\kappa}{\beta}
\binom{\kappa}{\eta_x+\eta_y-\beta}
\\&=(\rho/\kappa)^{\eta_x+\eta_y}(1-\rho/\kappa)^{2\kappa-(\eta_x+\eta_y)} \binom{2\kappa}{\eta_x+\eta_y}.
\end{aligned}
\end{equation*}
Finally, 
\begin{equation*}
c^{\mathrm{Binomial}}_\alpha(\eta_x,\eta_y)
= \binom{\kappa}{\eta_x-\alpha}\binom{\kappa}{\eta_y+\alpha} 
\binom{2\kappa}{\eta_x+\eta_y}^{-1},
\end{equation*}
which follows from the fact that the probability mass function of the hypergeometric distribution is normalized to be one. 
As a consequence, the rate $c_\alpha$ is normalized as required in \eqref{eq:redistribution_rate_normalization}. 
Therefore, again we know that the process preserves the degree up to one by \eqref{eq:redistribution_degree-one_preservation}. 
On the other hand, notice that 
\begin{equation*}
\begin{aligned}
M_2^{\mathrm{Binomial}}(\eta_x,\eta_y)
&=\sum_{\alpha=\max\{\eta_x-\kappa,-\eta_y\} }^{\min\{\eta_x, \kappa-\eta_y\} } 
\alpha^2 c^{\mathrm{Binomial}}_\alpha(\eta_x,\eta_y) \\
&\quad= 
\binom{2\kappa}{\eta_x+\eta_y}^{-1} 
\sum_{\beta=\max\{0,\eta_x+\eta_y-\kappa\}}^{ \min\{\eta_x+\eta_y,\kappa \} }
(\beta-\eta_x)^2 
\binom{\kappa}{\beta}
\binom{\kappa}{\eta_x+\eta_y-\beta} 
\\
&\quad=
\frac{(\eta_x+\eta_y)(2\kappa-\eta_x-\eta_y)}{4(2\kappa-1)} 
- (\eta_x-\eta_y) \frac{\eta_x+\eta_y}{2} 
+ \eta_x^2 - \frac{(\eta_x+\eta_y)^2}{4} \\
&\quad=\frac{\kappa-1}{4\kappa-2}(\eta_x^2 
+\eta_y^2 )
-\frac{\kappa}{2\kappa-1} \eta_x\eta_y
+ \frac{\kappa}{4\kappa-2}(\eta_x+\eta_y) 
\end{aligned}
\end{equation*}
where we used a trivial identity 
\begin{equation*}
(\beta-\eta_x)^2 
= (\beta-\lambda)^2 -(\eta_x-\eta_y)\beta
+ \eta_x^2-\lambda^2 
\end{equation*}
with $\lambda=(\eta_x+\eta_y)/2$ and used the fact that the mean (resp. variance) of the hypergeometric distribution that is involved with the second line is given by $\lambda$ (resp. $\lambda(\kappa-\lambda)/(2\kappa-1)$).  
Hence, \eqref{eq:adjoint_probability_second_moment_quadratic} holds and degree-two polynomials are preserved as well. 
Here, let us comment that this model was originally introduced in \cite[Section 5]{carinci2013duality} as the thermalized version of the partial exclusion process, see below.

\subsubsection{Interacting particles} 
A  well known model that admits a binomial distribution as reversible measure is the so-called partial exclusion process (PEP) which has  the following Markov generator that acts on functions $f:\mathscr X_N\to\mathbb R$ as:
\begin{equation*}
L^{\mathrm{PEP}}f(\eta)
=\frac{1}{2} \sum_{x,y\in\mathbb T_N,\, |x-y|=1} 
\eta_x(\kappa-\eta_{y})
\big[ f(\eta+\delta_{y}-\delta_x) - f(\eta)\big]. 
\end{equation*}
Then, it is not hard to show that the product binomial distribution $\nu_\rho^{\mathrm{Binomial}}$ is an invariant measure of PEP, and by a direct computation, {it is easy to check} that the action of the generator preserves the degree up to two. 
Here, we note that the parameter $\kappa \in \mathbb{N}$ is interpreted as the maximum number of particles at each site.

\subsection{Negative-binomial distribution}
Let us take $S=\mathbb N_0$ for this case. 
Let $\mathfrak s>0$  and  $\nu_\rho^{\mathrm{NB}}$ be the product measure whose common marginal is given by the negative-binomial distribution with mean $\rho$ given by: 
\begin{equation*}
\nu_{\rho}^{\mathrm{NB}}(\eta_x=k) 
=p^{\mathrm{NB}}(k)
=
\frac{\Gamma(2\mathfrak s +k)}{k!\Gamma(2\mathfrak s)}   
\Big( \frac{2\mathfrak s}{2\mathfrak s+\rho}\Big)^{2\mathfrak s}
\Big( \frac{\rho}{2\mathfrak s+\rho}\Big)^{k} 
\end{equation*}
for each $k\in \mathbb N_0$. 
In this parametrization, we have $V^{\mathrm{NB}}(\rho)=\rho^2/(2\mathfrak s) + \rho$. 

\subsubsection{Redistribution} 
Now, let us describe the process with redistribution-type interaction associated with the negative binomial distribution. 
Then, the generator \eqref{eq:generator_redistribution} is given on $f:\mathscr X_N\to\mathbb R$ as 
\begin{equation*}
\nabla_{x,y}^{\mathrm{NB}} f(\eta)=
\sum_{\alpha=-\eta_y}^{\eta_x} 
c^{\mathrm{NB}}_\alpha(\eta_x,\eta_y) 
[f(\eta-\alpha \delta_x+\alpha\delta_y) - f(\eta)] .
\end{equation*}
Now, note that 
\begin{equation*}
\begin{aligned}
Z^{\mathrm{NB}}(\eta_x+\eta_y)
&= \sum_{\alpha=-\eta_y}^{\eta_x} p^{\mathrm{NB}}(\eta_x-\alpha)p^{\mathrm{NB}}(\eta_y+\alpha)
= \sum_{\beta=0}^{\eta_x+\eta_y} p^{\mathrm{NB}}(\beta)p^{\mathrm{NB}}(\eta_x+\eta_y-\beta)\\
&=\Big( \frac{2\mathfrak s}{2\mathfrak s+\rho}\Big)^{4\mathfrak s}
\Big( \frac{\rho}{2\mathfrak s+\rho}\Big)^{\eta_x+\eta_y} 
\sum_{\beta=0}^{\eta_x+\eta_y} \frac{\Gamma(2\mathfrak s+\beta)}{\beta!\Gamma(2\mathfrak s)} \frac{\Gamma(2\mathfrak s+\eta_x+\eta_y-\beta)}{(\eta_x+\eta_y-\beta)!\Gamma(2\mathfrak s)} \\
&=\Big( \frac{2\mathfrak s}{2\mathfrak s+\rho}\Big)^{4\mathfrak s}
\Big( \frac{\rho}{2\mathfrak s+\rho}\Big)^{\eta_x+\eta_y} 
\binom{\eta_x+\eta_y+4\mathfrak s-1}{\eta_x+\eta_y}
\end{aligned}
\end{equation*} and then  the non-negative rate takes the form 
\begin{equation*}
c^{\mathrm{NB}}_\alpha(\eta_x,\eta_y) 
= \frac{ B(2\mathfrak s+\eta_x-\alpha,2\mathfrak s+\eta_y+\alpha) }{B(2\mathfrak s,2\mathfrak s)}
\frac{(\eta_x+\eta_y)!}{(\eta_x-\alpha)!(\eta_y+\alpha)!} \;.
\end{equation*}

Therefore, the rate $c_\alpha^{\mathrm{NB}}$ is normalized to satisfy \eqref{eq:redistribution_rate_normalization}, so that it preserves degree-one terms by \eqref{eq:redistribution_degree-one_preservation}. 
For degree-two terms, note that 
\begin{equation*}
\begin{aligned}
M_2^{\mathrm{NB}}(\eta_x,\eta_y)
&= \sum_{\alpha=-\eta_y}^{\eta_x} \alpha^2 
c_\alpha^{\mathrm{NB}}(\eta_x,\eta_y) \\
&= \sum_{\beta=0}^{\eta_x+\eta_y} 
(\eta_x-\beta)^2 
\frac{ B(2\mathfrak s+\beta, 2\mathfrak s+\eta_x+\eta_y-\beta) }{B(2\mathfrak s,2\mathfrak s)}
\frac{(\eta_x+\eta_y)!}{\beta!(\eta_x+\eta_y-\beta)!}\\
&= \frac{2\mathfrak s+1}{2(4\mathfrak s+1)}(\eta_x^2 + \eta_y^2)
- \frac{2\mathfrak s}{4\mathfrak s+1} \eta_x\eta_y
+ \frac{\mathfrak s}{4\mathfrak s+1}(\eta_x+\eta_y).
\end{aligned}
\end{equation*}
Here we used the fact that the distribution of $q_\alpha d\alpha$ can be written after the change of variable as a Beta-Binomial distribution $\mathrm{BB}(n,a,b)$ with $n= \eta_x+\eta_y$ and $a=b= 2\mathfrak{s}$. 
Note that mean of this distribution is given by 
\begin{equation*}
    n \frac{a}{a+b} = 2(\eta_x + \eta_y)
\end{equation*}
and the second moment is given by 
\begin{equation*}
\frac{n a}{(a+b)^{2}}
\left(
    b\,\frac{a+b+n}{a+b+1}
    + n a
\right) = (\eta_x + \eta_y)^2 \dfrac{2\mathfrak{s} +1}{2(4\mathfrak{s} +1)} + (\eta_x + \eta_y) \dfrac{\mathfrak{s}}{4\mathfrak{s} +1}. 
\end{equation*} 
Hence, the assertion \eqref{eq:adjoint_probability_second_moment_quadratic} is satisfied and thus the generator preserves degree-two terms as well.

\subsubsection{Interacting particles} 
Here, as an interacting particle model, we consider the symmetric inclusion process~(SIP), which was introduced in~\cite{giardina2007duality} as a discrete dual of the Brownian momentum process. 
The generator of the process is given on functions $f:\mathscr X_N\to\mathbb R$ by 
\begin{equation*}
L^{\mathrm{SIP}}f(\eta)
=\frac{1}{2} 
\sum_{x,y\in\mathbb T_N,|x-y|=1}
\eta_x (2\mathfrak s+\eta_{y})\big[f(\eta-\delta_x+\delta_y) - f(\eta)\big] 
\end{equation*}
 where $\delta_x$ denotes the configuration such that there is one particle at site $x$, whereas no particle at others. 
Then, the process generated by $L^{\mathrm{SIP}}$ is reversible with respect to the product measure $\nu^{\mathrm{NB}}_\rho$ with spin parameter $\mathfrak s>0$.   
With some computations, it is straightforward to check that all the assumptions that we imposed in general for the degree-preserving process hold true for SIP.

\subsubsection{Harmonic model}
\label{harmonicmodel}
Lastly, we give another example of a family of particle systems of zero-range type which was originally introduced via a mapping of the non-compact Heisenberg XXX chain in one dimension with open boundary,  see \cite{frassek2019non}.
Its invariant measure is given by a product of negative-binomial distributions,  see also \cite[Section 2.2.2]{franceschini2026hydrodynamic}. 
The generator of the harmonic models labeled by the spin parameter $\mathfrak s>0$ is given by \eqref{eq:generator_redistribution} with $\nabla_{x,y}=\nabla^{\mathrm{Harm}}_{x,y}$ where $\nabla^{\mathrm{Harm}}_{x,y}$ is defined for each $f:\mathscr X_N\to \mathbb R$ by  
\begin{equation*}
\nabla^{\mathrm{Harm}}_{x,y} f(\eta) 
= \sum_{\alpha=1}^{\eta_x} \dfrac{\Gamma(\eta_x +1) \Gamma(\eta_x - \alpha + 2\mathfrak s)}{\alpha \Gamma(\eta_x - \alpha +1 ) \Gamma(\eta_x + 2\mathfrak s)}  
\big[ f(\eta -\alpha\delta_x + \alpha\delta_{y}) - f(\eta)\big] . 
\end{equation*}
In other words, $\alpha$ particles jump from site $x$ to site $y$ with rate 
\begin{equation}
c^{\mathrm{Harm}}_\alpha (\eta_x,\eta_y)
= \dfrac{\Gamma(\eta_x+1) \Gamma(\eta_x - \alpha + 2\mathfrak s)}{ \alpha \Gamma(\eta_x-\alpha+1) \Gamma(\eta_x + 2\mathfrak s)} \mathbf{1}_{\lbrace \alpha \leq \eta_x \rbrace}. 
\end{equation}
Similarly to the other models, the Harmonic model exhibits reversibility, and thus stationarity, with respect to a product measure $\nu_\rho^{\mathrm{NB}}$. Moreover, a straightforward computation shows that the harmonic model preserves the degree up to two, see \cite{franceschini2026hydrodynamic} for more details.

\subsection{Generalized hyperbolic secant (GHS) distribution} 
\label{GHS_Section}
Here, we recall from \cite[Chapter 3]{fischer2013generalized} the definition and basic properties of generalized hyperbolic secant~(GHS) distribution, or the Meixner distribution in some literature.  
The density of the GHS distribution with parameters $r>0$ and $\theta\in\mathbb R$ is given as follows: 
\begin{equation*}
p_{r,\theta}^{\mathrm{GHS}}(\eta) 
= e^{\theta \eta + r \log \cos \theta} 
\frac{2^{r-2}}{\pi \Gamma(r)} 
\Big| \Gamma\Big(\frac{r}{2}+ i\frac{\eta}{2}\Big)\Big|^2 
\end{equation*}
for $\eta\in\mathbb R$, where $i=\sqrt{-1}$ denotes the imaginary unit.  
Here, notice that the density of the GHS distribution is expressed with an infinite product since 
\begin{equation*}
\bigg| \frac{\Gamma\big(\frac{r}{2}+ i\frac{\eta}{2}\big)}{\Gamma(\frac{r}{2}) } \bigg|^2 
= \prod_{j=0}^\infty \Big( 1 +\frac{\eta^2}{(r+2j)^2} \Big)^{-1} ,
\end{equation*}
see \cite[6.1.25, p.256]{abramowitz1965handbook}.  
In particular, using the relation $|\Gamma(1/2+i\eta)|^2=\pi/\cosh(\pi \eta)$, it turns out that when $r=1$ and $\theta=0$, it reduces to the hyperbolic secant distribution:
\begin{equation*}
 p^{\mathrm{GHS}}_{1,0}(\eta)
={\frac {1}{2}}\operatorname {sech} \left( {\frac {\pi \eta}{2}} \right)  
= (e^{\pi \eta/2} + e^{-\pi \eta/2})^{-1}. 
\end{equation*}
The density $p^{\mathrm{GHS}}_{r,0}$ with generic ${r > 0}$ can be interpret as the ``$r$-th'' convolution of $p_{1,0}^{\mathrm{GHS}}$, and particularly we have an identity $p^{\mathrm{GHS}}_{r,0}*p^{\mathrm{GHS}}_{r,0}(\eta)=p^{\mathrm{GHS}}_{2r,0}(\eta)$ as it is clear from the form of the moment generating function below, and the parameter $\theta$ is introduced via the Esscher transformation, see \cite[Section 5]{morris1982natural} for more details.   
It is known that the  mean (resp. variance) of the GHS distribution with the above parameterization is given by $r\tan \theta$ (resp. $r\tan^2\theta+r$). 
In order to have an association with the density, let us take $\theta=\mathrm{arctan} (\rho/r)$ an let $\nu^{\mathrm{GHS}}_\rho$ be the product measure whose common marginal is given as GHS with this reparameterization. 
Then we know that $E_{\nu_\rho^{\mathrm{GHS}}}[\eta_0]=\rho$ and the variance function in this case is given by $V^{\mathrm{GHS}}(\rho)=\rho^2/r + r$. 
Moreover, note that the moment generating function of GHS is given by 
\begin{equation}
\label{eq:GHS_MGF_definition}
\mathscr M^{\mathrm{GHS}}_{r,\theta}(t)
\coloneqq E_{\nu^{\mathrm{GHS}}_\rho} \big[e^{t\eta_0} \big] 
= \left( \frac{\cos \theta}{\cos(\theta+t)} \right)^r 
\end{equation}
with $\theta=\mathrm{arctan}(\rho/r)$, see \cite[Section 3, Eq. (3.8)]{fischer2013generalized}.

\subsubsection{Redistribution}
Now, let us describe the process associated with the GHS distribution. To the best of our knowledge, there is no references in the literature regarding this model.
The infinitesimal generator is given by \eqref{eq:generator_redistribution} with the operator $\nabla ^{x,y}$ acting on functions $f:\mathscr X_N\to\mathbb R$ as
\begin{equation*}
\nabla_{x,y}f(\eta) 
= \int_{-\infty}^\infty 
c_\alpha^{\mathrm{GHS}}(\eta_x,\eta_y) 
[f(\eta-\alpha \delta_x + \alpha\delta_y) - f(\eta)]d\alpha.
\end{equation*}
We also note that the normalizing factor defined in \eqref{eq:partition_function_redistribution_def} reads as
\begin{equation*}
Z^{\mathrm{GHS}}_{r,\theta}(\eta_x+\eta_y)
= \big( p^{\mathrm{GHS}}_{r,\theta} *
p^{\mathrm{GHS}}_{r,\theta}\big) (\eta_x+\eta_y)
=p^{\mathrm{GHS}}_{2r,\theta} (\eta_x+\eta_y)
\end{equation*}
where the last identity follows from the form of the moment generating function \eqref{eq:GHS_MGF_definition}. And the  non-negative rate $c_\alpha^{\mathrm{GHS}}(\eta_x,\eta_y)$ is defined by 
\begin{equation*}
c^{\mathrm{GHS}}_\alpha(\eta_x,\eta_y)
= \frac{1}{Z^{\mathrm{GHS}}_{r,\theta}  (\eta_x+\eta_y)} 
p^{\mathrm{GHS}}_{r,\theta}(\eta_x-\alpha)
p^{\mathrm{GHS}}_{r,\theta}(\eta_y+\alpha).
\end{equation*}
Recall that from \eqref{eq:redistribution_rate_normalization} and \eqref{eq:redistribution_degree-one_preservation}, we already know that degree-one terms are preserved under the action of the generator. 
For degree-two terms, let us check that the second moment $M_2^{\mathrm{GHS}}$ is a second-order polynomial.  
To that end, note that 
\begin{equation*}
\begin{aligned}
M_2^{\mathrm{GHS}}(\eta_x,\eta_y)
&= \frac{1}{Z^{\mathrm{GHS}}_{r,\theta}(\eta_x+\eta_y) }  
\int_{-\infty}^\infty \alpha^2
p^{\mathrm{GHS}}_{r,\theta} (\eta_x-\alpha)
p^{\mathrm{GHS}}_{r,\theta} (\eta_y+\alpha)d\alpha \\
&=\eta_x\eta_y 
+ (\eta_x-\eta_y)M_1^{\mathrm{GHS}}(\eta_x,\eta_y) 
- \frac{F^{\mathrm{GHS}}_{r,\theta}(\eta_x+\eta_y)}{Z^{\mathrm{GHS}}_{r,\theta}(\eta_x+\eta_y)} 
\end{aligned}
\end{equation*}
where 
\begin{equation*}
F^{\mathrm{GHS}}_{r,\theta}(s)
= \int_{-\infty}^\infty 
\beta(\eta_x+\eta_y-\beta)
p^{\mathrm{GHS}}_{r,\theta} (\beta)
p^{\mathrm{GHS}}_{r,\theta} (s-\beta) d\beta.
\end{equation*}
Above, we used a trivial identity 
\begin{equation*}
\alpha^2 
= \eta_x\eta_y 
+(\eta_x-\eta_y)\alpha
-(\eta_x-\alpha)(\eta_y+\alpha) 
\end{equation*}
and then applied a change of variable $\beta=\eta_x-\alpha$. 
Here recall that we already know that the first moment is degree-one: $M^{\mathrm{GHS}}_1(\eta_x,\eta_y)=(\eta_x-\eta_y)/2$.   
To compute $F^{\mathrm{GHS}}_{r,\theta}(\eta_x,\eta_y)$, let us apply the Laplace transformation, which is denoted by $\mathcal L$. 
Note that 
\begin{equation*}
\begin{aligned}
\mathcal LF_{r,\theta}(t)
&\coloneqq \int_{-\infty}^\infty
\int_{-\infty}^\infty
\beta(s-\beta)
p^{\mathrm{GHS}}_{r,\theta} (\beta)
p^{\mathrm{GHS}}_{r,\theta} (s-\beta) e^{ts} d\beta ds \\
&= \bigg( \int_{-\infty}^\infty
\beta p^{\mathrm{GHS}}_{r,\theta}(\beta)e^{t\beta} d\beta \bigg)^2 
= \Big( \frac{d}{dt}\mathscr M^{\mathrm{GHS}}_{t,\theta}(t) \Big)^2 
= r^2 \Big(\frac{\cos\theta}{\cos(\theta+t)} \Big)^{2r} \tan^2(\theta+t) \\
&=r^2 \Big( 1+ \frac{\rho^2}{r^2}\Big) \mathscr M^{\mathrm{GHS}}_{2r+2,\theta}(t)
-r^2 \mathscr M^{\mathrm{GHS}}_{2r,\theta}(t).
\end{aligned}
\end{equation*}
Therefore, from the Laplace inversion, this immediately yields 
\begin{equation*}
F^{\mathrm{GHS}}_{r,\theta}(s)
= (r^2+\rho^2)p^{\mathrm{GHS}}_{2r+2,\theta}(s)
-r^2 p^{\mathrm{GHS}}_{2r,\theta}(s).
\end{equation*}
However, here we note that  
\begin{equation*}
\begin{aligned}
\frac{p^{\mathrm{GHS}}_{2r+2,\theta}(s)}{p^{\mathrm{GHS}}_{2r,\theta}(s)}
&=\frac{e^{(2r+2)\log \cos\theta}}{e^{2r \log \cos\theta}} \frac{4\Gamma(2r)}{\Gamma(2r+2)}
\bigg| \frac{\Gamma(\tfrac{2r+2}{2}+i\tfrac{s}{2})}{ \Gamma(\tfrac{2r}{2}+i\tfrac{s}{2})} \bigg|^2 \\
&= \cos^2\theta \frac{2}{r(2r+1)} \Big( r^2 + \frac{s^2}{4}\Big) 
= \frac{r(4r^2 + s^2)}{2(2r+1)(r^2 +\rho^2)}
\end{aligned}
\end{equation*}
where we used an elementary identity of the Gamma function $\Gamma(z+1)=z\Gamma(z)$. 
Since the previous display is a quadratic function of $s$ and so is $F^{\mathrm{GHS}}_{r,\theta}(s)/Z^{\mathrm{GHS}}_{r,\theta}(s)$, we conclude that the action of the generator preserves degree-two terms.

\begin{Remark}
\label{rem:GHS_not_exists}
It is not possible to identify a diffusion process whose reversible product measure is the GHS distribution and for which conditions \textnormal{(A6.1)} and \textnormal{(A6.2)} on degree-preserving holds. 
This is because such conditions require a restrictive form of coefficients of each derivative operator in the generator of the diffusion process, and thus the invariance of the GHS distribution for each marginal is violated.   
\end{Remark}

\section{Random walk representation and second-moment bounds} \label{sec4}
Here let us give some key estimates, which will be used ahead in our proof of the hydrodynamic limit. 
In what follows, we always assume  both Assumptions  \ref{assmp:generator_element_basic} and \ref{assumption:second_order_polymonial}.

\subsection{Dynkin's martingales} 
\label{subsec4.1} 
The building blocks for our proof of the hydrodynamic limit (\cref{thm:hydrodynamic_general_main})  are the  Dynkin's martingales, which, in this context, read as follows. 
Recall from \eqref{eq:empirical_measure_definition} the definition of the empirical measure. 
From Dynkin's martingale formula~(see \cite[Lemma A1.5.1]{kipnis1999scaling} or \cite[Proposition VII.1.6]{revuz2013continuous}), for any $G \in C^2(\mathbb T)$, the processes  
\begin{equation}\label{eq:dynkin_1}
M^N_t(G)
= \pi^N_t(G)
- \pi^N_0(G)
-\int_0^t N^2L\pi^N_s(G)ds 
\end{equation}
and $M^N_t(G)^2-\langle M^N(G)\rangle_t$ are  mean-zero martingales with respect to the natural filtration, where 
\begin{equation} \label{eq:dynkin_2}
\langle M^N(G)\rangle_t 
= \int_0^t \Upsilon^N_s(G)ds 
\end{equation}
and the carr\'e du champs operator is defined by 
\begin{equation*}
\Upsilon^N_t(G)= N^2 L \langle\pi_t^N,G \rangle^2 - 2N^2 \langle\pi_t^N,G \rangle L \langle\pi_t^N,G \rangle .
\end{equation*}

Let us start from the observation that the system is of gradient type, with constant diffusion coefficient.  
For $x\in\mathbb T_N$, let  $W_{x,x+1}=L_{x,x+1}\eta_x$ be the instantaneous current between sites $x$ and $x+1$. As we shall see below, all our  models are of gradient type, that is, the current can be  written as a gradient of some local function.

\begin{lemma}\label{lem:lap}
For all $x\in\mathbb T_N$, we have $W_{x,x+1}=D(\eta_{x+1}-\eta_x)$, so that the model is of gradient type. Moreover, we have that 
$
L \eta_x =D \Delta \eta_x
$
where $\Delta$ denotes the discrete Laplacian {though now acting on functions of the state-space.}
\end{lemma}
\begin{proof}
The computation of the instantaneous current follows directly from \cref{lem:gradient_condition}. 
As a consequence, we can write
\begin{equation}
\label{diff}
\begin{aligned}
L \eta_x 
&= L_{x,x+1} \eta_x + L_{x-1,x} \eta_x \\
&= D(\eta_{x+1} - \eta_x) + D(\eta_{x-1} - \eta_x) =  D
(\eta_{x+1} - 2\eta_x + \eta_{x-1}) = D \Delta \eta_x \;.
\end{aligned}
\end{equation}
\end{proof}

Next, we note that we have the following explicit representation of the carr\'e du champs operator.  

\begin{lemma}
\label{lem:qua_var}
We have that 
\begin{equation*}
\Upsilon^N_t(G) 
= \dfrac{1}{N^2}\sum_{x \in \mathbb{T}_N} \left( \nabla^+_N  G \left( \tfrac{x}{N} \right) \right)^2 \left[ D \left( \eta_x (t)- \eta_{x+1}(t) \right)^2 
- L_{x,x+1}(\eta_x(t) \eta_{x+1}(t)) \right] .
\end{equation*}
Above and in what follows, $\nabla^\pm_N G(x/N)=N(G((x\pm 1)/N) -G(x/N))$. 
\end{lemma}
\begin{proof}
We can explicitly compute
\begin{align} 
\label{integrand}
\Upsilon^N_s(G) 
& = \sum_{x,y \in \mathbb{T}_N} 
G\left( \tfrac{x}{N}\right)
G\left( \tfrac{y}{N}\right) \left[L \eta_x(s) \eta_y(s) - 2 \eta_x(s) L \eta_y(s) \right] \\ 
& =  \sum_{x \in \mathbb{T}_N} 
G\left( \tfrac{x}{N}\right)
\left[L \eta_x(s)^2 - 2 \eta_x(s) L \eta_x(s) \right] \nonumber \\ 
&\quad+ \sum_{x \in \mathbb{T}_N} 
2 G\left( \tfrac{x}{N}\right)
G\left( \tfrac{x+1}{N}\right)
\left[L (\eta_x(s)\eta_{x+1}(s)) -  \eta_x(s) L \eta_{x+1}(s) -  \eta_{x+1}(s) L \eta_{x}(s)  \right] \nonumber
\end{align}
and we are left to compute the second-order terms.
For the first one we have 
\begin{align*}
L \eta_x^2 = L_{x,x+1}\eta_x^2 + L_{x-1,x}\eta_x^2
\end{align*}
where 
\begin{align*}
L_{x,x+1} \eta_x^2 &= L_{x,x+1} \left( \eta_x^2 + \eta_x \eta_{x+1} \right) - L_{x,x+1} \left( \eta_x \eta_{x+1} \right) \\ & = \left( \eta_x + \eta_{x+1} \right) L_{x,x+1} \eta_x - L_{x,x+1}\left( \eta_x \eta_{x+1} \right) \\ & = 
D \left( \eta_{x+1}^2 - \eta_x^2 \right) -  L_{x,x+1}\left( \eta_x \eta_{x+1} \right)
\end{align*}
and similarly
\begin{align*}
L_{x-1,x} \eta_x^2 = D \left( \eta_{x-1}^2 - \eta_x^2 \right) -  L_{x-1,x}\left( \eta_{x-1} \eta_{x} \right) .
\end{align*}
Above, we used the fact that $L_{x,x+1} = L_{x+1,x}$ and 
\begin{equation*}
L_{x,x+1} \big( F(\eta)G(\eta_x,\eta_{x+1})\big) 
= F(\eta)L_{x,x+1}G(\eta_x,\eta_{x+1}) 
\end{equation*}
for any $F:\mathscr X_N\to\mathbb R$ which is a function of $\eta_y, y \neq x, x+1$ and $\eta_x+\eta_{x+1}$.
This, together with \eqref{diff}, leads to 
\begin{align*}
L \eta_x^2 -2\eta_x L \eta_x 
&= D(\eta_{x+1}^2 - \eta_x^2) - L_{x,x+1}(\eta_x \eta_{x+1}) -2D\eta_x (\eta_{x+1} - \eta_x) \\
&\quad + D(\eta_{x-1}^2 - \eta_x^2) - L_{x-1,x}\eta_x \eta_{x-1} -2D\eta_x (\eta_{x-1} - \eta_x) \\
&=  D (\eta_{x+1} - \eta_x)^2 - L_{x,x+1}(\eta_x \eta_{x+1}) + D(\eta_{x-1} - \eta_x)^2 - L_{x-1,x}\eta_x \eta_{x-1}.
\end{align*}
For the last line of \eqref{integrand} we have that
\begin{align*}
L (\eta_x\eta_{x+1}) -  \eta_x L \eta_{x+1} -  \eta_{x+1} L \eta_{x} 
&= L_{x,x+1} (\eta_x \eta_{x+1}) - \eta_x L_{x,x+1} \eta_{x+1} - \eta_{x+1}L_{x,x+1} \eta_x \\
& = L_{x,x+1} (\eta_x \eta_{x+1}) - \eta_x D (\eta_x - \eta_{x+1}) - \eta_{x+1} D (\eta_{x+1} - \eta_{x}) \\
&= L_{x,x+1} (\eta_x \eta_{x+1}) -D (\eta_x - \eta_{x+1})^2
\end{align*}
where we used the fact that 
\begin{equation*}
\begin{aligned}
L_{x,x+1} \eta_{x+1} 
= L_{x,x+1} ( \eta_x + \eta_{x+1} - \eta_x) 
= -  L_{x,x+1} \eta_x = D (\eta_x - \eta_{x+1}) .
\end{aligned}
\end{equation*}
This allows us to rewrite the integrand part of the quadratic variation as
\begin{align*}
\Upsilon^N_t(G) 
& = \sum_{x \in \mathbb{T}_N} 
\left( G\left( \tfrac{x}{N}\right) - G\left(\tfrac{x+1}{N}\right) \right)^2 
\big[ D( \eta_x - \eta_{x+1})^2 - L_{x,x+1} (\eta_x \eta_{x+1}) \big](t) \\ &
= \dfrac{1}{N^2}\sum_{x \in \mathbb{T}_N} 
\left( \nabla^+_N G \left( \tfrac{x}{N} \right) \right)^2 \big[ D ( \eta_x - \eta_{x+1})^2 
- L_{x,x+1}(\eta_x \eta_{x+1}) \big](t) .
\end{align*}
\end{proof}

\subsection{Main estimates} 
\label{sebsec:main_estimates}
Now we derive a random walk representation for a two-point space correlation function.  
We show that the correlation decays, and consequently, we have a uniform second-order moment bound. 
The next estimate was crucial to show the hydrodynamic limit, see \cite[Appendix A]{franceschini2026hydrodynamic}, however, for the present paper, it is obvious from the second-order polynomial assumption ({Assumption \ref{assumption:second_order_polymonial}}), i.e. \textnormal{(A6.2)}, and that $|\eta_x| \le 1+ \eta_x^2$.  

\begin{lemma}
\label{K-estimate} 
There exists a constant $K>0$ such that for any $\eta$ and $x\in\mathbb T_N$
\begin{equation}
\big| D( \eta_x - \eta_{x+1})^2 
- L_{x,x+1}(\eta_x \eta_{x+1})\big|  
\le K(\eta_x^2+\eta_{x+1}^2 +1). 
\end{equation}
\end{lemma}

Next, let us show the following preliminary result.  

\begin{lemma}
\label{lem:special_case_condition_bernoulli} 
We have that $v_2 >-1$. 
\end{lemma}
\begin{proof}
Since $v_2 \le -1$ happens only when $\eta$ is distributed according to the Bernoulli product measure, namely when $\kappa = 1$ in the case of binomial distribution.   
In this special case, $\eta_0^2=\eta_0, \eta_1^2=\eta_1$ and therefore $L_{0,1}(\eta_0\eta_1)=\frac{1}{2}L_{0,1}(\eta_0+\eta_1)^2 -\frac{1}{2}L_{0,1} (\eta_0+\eta_1)=0$ which necessary implies $\mathsf{a} =0$.
\end{proof}

From now on, we assume $\mathsf{a} \neq 0$, so that from the previous result $v_2 >-1$ . 
For $x\in\mathbb{T}_N$, let $\rho_x(t)=\mathbb{E}_{N}[\eta_x(t)]$ be a discrete density profile at time $t$ and let us denote by $\overline\eta_x(t)=
\eta_x(t) - \rho_x(t)$ the centering. 
Here, we can easily check that $\rho_x(t)$ satisfies the discrete heat equation $(d/dt)\rho_x(t)=D\Delta\rho_x(t)$ for each $x\in\mathbb T_N$, which is a system of ordinary differential equations.  
Moreover, for $i \in \{1,\dots, \lfloor N/2 \rfloor \}$, let $\phi(t,i)$ be defined by 
\begin{equation}\label{eq:phi}
\phi(t,i)
\coloneqq \sum_{x\in\mathbb T_N} \mathbb{E}_{N}[\overline\eta_x(t) \overline\eta_{x+i}(t)]
\end{equation}
whereas 
\begin{equation}\label{eq:redefine}
\phi(t,0)
\coloneqq \frac{1}{v_2+1} \sum_{x\in\mathbb T_N}
\mathbb{E}_{N}\Big[ 
\eta_x(t)^2
-\big\{(v_2+1)\rho_x(t)^2+v_1\rho_x(t)+v_0\big\} \Big] ,
\end{equation}
in the above we subtract the second moment of the occupation variable with respect to the invariant measure. 
Below, we would like to show the uniform $L^2$-bound of the occupation variables. 

\begin{lemma}
\label{lem:correlation_estimate}
Assume that there exists some constant $C_0>0$ such that 
\begin{equation}
\label{assump:initial_condition_correlation}
\sup_{N\in\mathbb N} 
\max_{i=0,\ldots,\lfloor N/2\rfloor}|\phi(0,i)| < C_0. 
\end{equation}
Then, there exists a constant $ C=C(T)$ such that 
\begin{equation*}
\sup_{N\in\mathbb N}\sup_{0\le t\le T} \max_{i=0,\ldots,\lfloor N/2\rfloor}|\phi(N^2t,i)| < C. 
\end{equation*}
In particular, we have the second moment estimate 
\begin{equation}
\label{eq:moment_estimate} 
\sup_{N\in\mathbb N} \sup_{0\le t\le T}
\frac{1}{N} 
\mathbb E_N\bigg[\sum_{x\in\mathbb T_N} \eta_x^2(N^2t)\bigg]
<C. 
\end{equation} 
\end{lemma}
\begin{Remark}
Note that the assumption \eqref{assump:initial_condition_correlation} is immediately satisfied under the stronger condition \eqref{assump:initial_condition_moment}, which will be used to show boundedness of the empirical measure at the initial time. 
\end{Remark}

\begin{Remark}
We will show that the uniform bound in \eqref{eq:moment_estimate} can be proved using the representation of a one-dimensional random walk which, at all times, records the difference between two occupation variables. In the previous work \cite{franceschini2026hydrodynamic}, we used the attractiveness of the processes under the stronger assumption that the initial measure is stochastically dominated by the invariant one. This condition is now removed, and therefore we can also include processes that are not attractive, e.g., the symmetric inclusion process (SIP). The second-moment estimate can also be proved via stochastic duality, assuming the bound \eqref{assump:initial_condition_moment} at time zero. However, following this route, we could not include the new model of Section \ref{GHS_Section}, for which duality is still under investigation. All duality relations for the remaining models can be found in \cite{carinci2013duality, dualitybook}.  
\end{Remark}

To show \cref{lem:correlation_estimate}, let us compute the time evolution of the correlation function $\phi(t,i)$, which varies in parity of $N$.

\begin{lemma}
We have 
\begin{equation}
\label{eq:correlation_function_time_evolution}
\begin{aligned}
\frac{d}{dt}\phi(t,i)
&=  2D\Delta\phi(t,i)
\mathbf{1}_{i\neq 0,1,\lfloor N/2\rfloor}
+ p_N D\nabla^- \phi(t,i)\mathbf{1}_{i=\lfloor N/2\rfloor} \\
&\quad+\big\{ 2D\nabla^+\phi(t,1)
+2\mathsf{a} (v_2+1)\nabla^-\phi(t,1)
+ (\mathsf{a} (v_2+1)-D)\|\nabla^+\rho_\cdot(t)\|^2_{\ell^2(\mathbb T_N)} \big\} \mathbf{1}_{i=1} \\
&\quad+ \big\{ 4\mathsf{a} \nabla^+\phi(t,0) + (2D-2\mathsf{a} )\|\nabla^+\rho_\cdot(t)\|^2_{\ell^2(\mathbb T_N)} \big\}\mathbf{1}_{i=0} 
\end{aligned}
\end{equation}
where $p_N=4$ if $N$ is even, whereas $p_N=2$ if $N$ is odd, and we introduced the discrete derivatives $\nabla^\pm g(i)=g(i\pm 1) -g(i)$, and $\Delta g(i)=g(i+1)+g(i-1)-2g(i)$ for any real sequence $(g(i))_{i\in\mathbb Z}$. 
\end{lemma}
\begin{proof}
First, let us consider the case $i \neq 0,1$.
By Kolmogorov's forward equation, we have 
\begin{align*}
\frac{d}{dt}\phi(t,i) 
= \sum_{x\in\mathbb T_N}
\mathbb E_N[L (\eta_x\eta_{x+i})(t)]
-\sum_{x\in\mathbb T_N}\frac{d}{dt}\big(\rho_x(t)\rho_{x+i}(t)\big) . 
\end{align*}
Here, note that 
\begin{equation*}
L(\eta_x\eta_{x+i})
= (D\Delta \eta_x)\eta_{x+i} + \eta_x(D\Delta \eta_{x+i})
\end{equation*}
provided $i\neq 0,1$. 
Hence, using the fact that $\rho_x(t)$ satisfies the discrete heat equation, 
\begin{align*}
\frac{d}{dt}\phi(t,i)
& = 2D \sum_{x\in\mathbb T_N} 
\mathbb E_N\big[ \overline\eta_x(t) \overline{\eta}_{x+i-1}(t) + \overline\eta_x(t) \overline{\eta}_{x+i+1}(t) - 
2\overline\eta_x(t) \overline{\eta}_{x+i}(t) \big] =  2D \Delta \phi(t,i)
\end{align*}
if $i \neq 0,1,\lfloor N/2 \rfloor$. 
Next, if $i=\lfloor N/2 \rfloor$ and $N$ is even, then note that 
\begin{equation*}
\mathbb E_N\bigg[ \sum_{x\in\mathbb T_N} \overline\eta_x(t) \overline\eta_{x+i+1}(t)\bigg] 
= \mathbb E_N\bigg[ \sum_{x\in\mathbb T_N} \overline\eta_x(t) \overline\eta_{x+i-1}(t)\bigg] 
=\phi(t,i-1)
\end{equation*}
and thus 
\begin{equation*}
\frac{d}{dt}\phi(t,i) = 4D (\phi(t,i-1)-\phi(t,i)).
\end{equation*}
On the other hand, if $i=\lfloor N/2 \rfloor$ and $N$ is odd, note that 
\begin{equation*}
\mathbb E_N\bigg[ \sum_{x\in\mathbb T_N} \overline\eta_x(t) \overline\eta_{x+i+1}(t)\bigg] 
= \mathbb E_N \bigg[ \sum_{x\in\mathbb T_N} \overline\eta_x(t) \overline\eta_{x+i}(t)\bigg] 
=\phi(t,i). 
\end{equation*}
Thus, we have that 
\begin{equation*}
\frac{d}{dt}\phi(t,i) = 2D (\phi(t,i-1)-\phi(t,i)).
\end{equation*}
Next, we consider the case $i=1$.
Again by the Kolmogorov forward equation, we have that 
\begin{equation*}
\frac{d}{dt}\phi(t,1)
=\sum_{x\in\mathbb T_N}\mathbb E_N[L(\eta_x\eta_{x+1})(t)]
- \sum_{x\in\mathbb T_N} \frac{d}{dt}(\rho_x(t)\rho_{x+1}(t)). 
\end{equation*}
Here, note that  
\begin{equation*}
\begin{aligned}
L(\eta_x\eta_{x+1})
&=L_{x,x+1}(\eta_x\eta_{x+1})
+ (L_{x-1,x}\eta_x)\eta_{x+1}
+ \eta_x(L_{x+1,x+2}\eta_{x+1})\\
&=L_{x,x+1}(\eta_x\eta_{x+1})
+ D(\eta_{x-1}-\eta_x)\eta_{x+1}
+ D\eta_x(\eta_{x+2}-\eta_{x+1}) . 
\end{aligned}
\end{equation*}
This immediately yields 
\begin{equation*}
\begin{aligned}
\frac{d}{dt}\phi(t,1) 
&= \sum_{x\in\mathbb T_N}\mathbb E_N [(L_{x,x+1} (\eta_x\eta_{x+1}))(t)] \\
&\quad+ D \sum_{x\in\mathbb T_N}\mathbb E_N [(\eta_{x-1}(t)-\eta_x(t) )\eta_{x+1}(t)] 
+  D \sum_{x\in\mathbb T_N}\mathbb E_N [ \eta_x(t) (\eta_{x+2} (t)-\eta_{x+1}(t))] \\
&\quad- 2D \sum_{x\in\mathbb T_N}\rho_{x-1}(t)\rho_{x+1}(t) +2D \sum_{x\in\mathbb T_N}\rho_x(t)\rho_{x+1}(t) -D \sum_{x\in\mathbb T_N} (\rho_x(t)-\rho_{x+1}(t))^2 \\
&= \sum_{x\in\mathbb T_N}\mathbb E_N \big[ (L_{x,x+1} (\eta_x\eta_{x+1}))(t)\big] + 2D (\phi(t,2)- \phi(t,1)) -D \sum_{x\in\mathbb T_N} (\rho_x(t)-\rho_{x+1}(t))^2. 
\end{aligned}
\end{equation*}
Now, we compute the first term of the last expression.
By Assumption \ref{assumption:second_order_polymonial}, we have that 
\begin{equation}
\label{eq:computation_degree2_sum}
\begin{aligned}
& \sum_{x\in\mathbb T_N}\mathbb E_N [(L_{x,x+1} (\eta_x\eta_{x+1})) (t)]\\
&\quad= 
2\mathsf{a} \sum_{x\in\mathbb T_N}\mathbb E_N [\eta_x(t)^2] 
+ \mathsf{b}   \sum_{x\in\mathbb T_N}\mathbb E_N [\eta_x(t)\eta_{x+1}(t)] 
+ 2\mathsf{c}   \sum_{x\in\mathbb T_N}\rho_x(t) 
+ \mathsf{d}    N \\
&\quad= 
2\mathsf{a} \left((v_2+1)\phi(t,0) +(v_2+1)\sum_{x\in\mathbb T_N}\rho_x(t)^2 + v_1\sum_{x\in\mathbb T_N}\rho_x(t) + v_0N \right) \\
&\qquad+ \mathsf{b}   \sum_{x\in\mathbb T_N}\mathbb E_N [\eta_x(t)\eta_{x+1}(t)] + 2\mathsf{c}   \sum_{x\in\mathbb T_N}\rho_x(t) + \mathsf{d}    N \\
&\quad= 2\mathsf{a} (v_2+1)\phi(t,0) +2\mathsf{a}  (v_2+1)\sum_{x\in\mathbb T_N}\rho_x(t)^2   - 2\mathsf{a}  (v_2+1)\sum_{x\in\mathbb T_N}\mathbb E_N [\eta_x(t)\eta_{x+1}(t)]  \\
&\quad = 2\mathsf{a} (v_2+1)(\phi(t,0) -\phi(t,1)) 
+\mathsf{a}  (v_2+1)\sum_{x\in\mathbb T_N}(\rho_x(t)-\rho_{x+1}(t))^2.
\end{aligned}
\end{equation}
Above we used the definition of $\phi(t,0)$ and the relations in \eqref{systemvs}. 
The last display deduces the expression in the assertion. 
Finally, let us consider the case $i=0$.
Recalling again the definition of $\phi(t,0)$, we have that 
\begin{equation*}
\frac{d}{dt}\phi(t,0)
= \frac{1}{v_2+1}\sum_{x\in\mathbb T_N}\mathbb E_N[(L \eta_x^2)(t)]
-\frac{1}{v_2+1}\frac{d}{dt}\sum_{x\in\mathbb T_N} 
\big( (v_2+1)\rho_x(t)^2+v_1\rho_x(t)+v_0\big). 
\end{equation*}
Moreover, recall that we have the following identities: 
\begin{equation*}
\begin{aligned}
&L_{x,x+1}\eta_x^2
= D(\eta_{x+1}^2-\eta_x^2)-L_{x,x+1}(\eta_x\eta_{x+1}),\\
&L_{x-1,x}\eta_x^2
= D(\eta_{x-1}^2-\eta_x^2)-L_{x-1,x}(\eta_{x-1}\eta_{x}).
\end{aligned}
\end{equation*}
Thus, using the fact that the telescopic sum vanishes, we have that 
\begin{align*}
\frac{d}{dt}\phi(t,0) 
&= \frac{1}{v_2+1}\sum_{x\in\mathbb T_N}
\mathbb E_N\big[ (L_{x-1,x} \eta_x^2)(t)
+ (L_{x,x+1} \eta_x^2)(t)\big]  
- \frac{d}{dt} \sum_{x\in\mathbb T_N} \rho_x(t)^2 \\
&= \frac{1}{v_2+1}\sum_{x\in\mathbb T_N}\mathbb E_N[D(\eta_{x+1}^2(t)-\eta_x^2(t))-(L_{x,x+1}(\eta_x\eta_{x+1}))(t)] \\
&\quad+ \frac{1}{v_2+1}\sum_{x\in\mathbb T_N}\mathbb E_N[D(\eta_{x-1}^2(t)-\eta_x^2(t))-(L_{x-1,x}(\eta_{x-1}\eta_x))(t)] \\
&\quad - 2D \sum_{x\in\mathbb T_N} \rho_x(t) \Delta \rho_{x}(t)
\\
&= - \frac{2}{v_2+1}\sum_{x\in\mathbb T_N}\mathbb E_N [(L_{x,x+1} (\eta_x\eta_{x+1}))(t)] +2D \sum_{x\in\mathbb T_N} (\rho_x(t)-\rho_{x+1}(t))^2.
\end{align*}
Hence, applying the computation \eqref{eq:computation_degree2_sum} for the case $i=1$, we have that 
\begin{align*}
\frac{d}{dt}\phi(t,0) 
=4\mathsf{a} (\phi(t,1) -\phi(t,0))
+(2D-2\mathsf{a} ) \sum_{x\in\mathbb T_N} (\rho_x(t)-\rho_{x+1}(t))^2.
\end{align*}
\end{proof}

This allows us to use the random walk representation and estimate $\phi(t,0)$ uniformly in $t$ and $N$. Additionally, note that this random walk depends only on $v_2, D$ and $\mathsf{a}$.  

\subsection{Proof of \texorpdfstring{\cref{lem:correlation_estimate}}{correlation estimate}}  
\label{subsec4.2.1}
Let $\mathbb I_N=\{0,1,\ldots,\lfloor N/2\rfloor\}$.  
Here let us write the time evolution of the correlation function given in \eqref{eq:correlation_function_time_evolution} in the following way:
\begin{equation}
\label{eq:correlation_function_time_evolution_general}
\frac{d}{dt}\phi(t,i)
= \mathscr L\phi(t,i)
+ \mathfrak g(t,i)
\end{equation}
where the operator $\mathscr L$, which is acting on the discrete space variable $i$ in the last expression, is defined by 
\begin{equation*}
\begin{aligned}
\mathscr LG(i)
&=2D\Delta G(i)\mathbf{1}_{i\neq 0,1,\lfloor N/2\rfloor} 
+ p_N D\nabla^-G(i)\mathbf{1}_{i=\lfloor N/2\rfloor} \\
&\quad+ \big[2D\nabla^+G(i) +2\mathsf{a}  (v_2+1) \nabla^-G(i)\big]\mathbf{1}_{i=1}
+ 4\mathsf{a} \nabla^+G(i)\mathbf{1}_{i=0} 
\end{aligned}
\end{equation*}
for any $\{G(i)\}_{i\in\mathbb I_N}$ and the reminder term $\mathfrak g(t,i)$ is defined by 
\begin{equation*}
\mathfrak g(t,i)
= (\mathsf{a} (v_2+1)-D) \|\nabla^+\rho_\cdot(t)\|^2_{\ell^2(\mathbb T_N)} \mathbf{1}_{i=1}
+ (2D-2\mathsf{a} )\|\nabla^+\rho_\cdot(t)\|^2_{\ell^2(\mathbb T_N)} \mathbf{1}_{i=0} .
\end{equation*}
Now, from \eqref{eq:correlation_function_time_evolution_general}, by Duhamel's principle, we have the following representation of the correlation function: 
\begin{equation*}
\phi(t,i)
= \mathbf E_i
\bigg[ \phi(0,X_t^{(1)})
+ \int_0^t \mathfrak g(t-s,X_s^{(1)})ds \bigg]
\end{equation*}
where $\{X^{(1)}_t\}_{t\ge 0}$ is a random walk on $\mathbb I_N$ with infinitesimal generator $\mathscr L$, and we denote by $\mathbf P_i$ the associated probability measure of the random walk, provided it starts from $i\in\mathbb I_N$, and we write the expectation with respect to $\mathbf P_i$ by $\mathbf E_i$. 
Recalling the definition of $\mathfrak g$, we have that 
\begin{equation*}
\begin{aligned}
\phi(t,i)
&= \mathbf E_i[\phi(0,X^{(1)}_t)]\\
&+ \int_0^t \sum_{j\in\mathbb I_N} 
\big[ (\mathsf{a} (v_2+1)-D)\mathbf 1_{j=1}
+ (2D-2\mathsf{a} ) \mathbf 1_{j=0} \big] \|\nabla^+\rho_\cdot(t-s)\|^2_{\ell^2(\mathbb T_N)}  \mathbf P_i(X^{(1)}_s=j)
ds .
\end{aligned}
\end{equation*}
Now, to be in the diffusive time scaling, inserting $t=tN^2$, a change of variable yields 
\begin{equation*}
\begin{aligned}
\phi(tN^2&,i)
= \mathbf E_i[\phi(0,X^{(1)}_{tN^2})]\\
&+ \int_0^t \sum_{j\in\mathbb I_N} 
\big[ (\mathsf{a} (v_2+1)-D)\mathbf 1_{j=1}
+ (2D-2\mathsf{a} ) \mathbf 1_{j=0} \big] \|\nabla^{+}_N \rho^N_\cdot(t-s)\|^2_{\ell^2(\mathbb T_N)} \mathbf P_i(X^{(1)}_{sN^2}=j)
ds 
\end{aligned}
\end{equation*}
where we set $\rho^N_\cdot(t)=\rho_\cdot(tN^2)$. 
Hence, taking the supremum over $i$ and $t$, we have the bound 
\begin{equation*}
\sup_{0\le s\le t} \|\phi(sN^2,\cdot)\|_{\ell^\infty(\mathbb I_N)} 
\le \|\phi(0,\cdot)\|_{\ell^\infty(\mathbb I_N)}
+ C \sup_{0\le s\le t} \|\nabla^+_N\rho^N_\cdot(s)\|^2_{\ell^2(\mathbb T_N)}
\max_{i\in\mathbb I_N} \mathcal T^{(1)}_N(t,i)
\end{equation*}
with some $C=C(\mathsf{a} ,v_2,D)>0$, where 
\begin{equation*}
\mathcal T^{(1)}_N(t,i)
=\int_0^t \mathbf P_i(X^{(1)}_{sN^2}\in \{0,1\}) ds 
\end{equation*}
is the local time that the random walk $\{X^{(1)}_{tN^2};t\ge 0\}$ starting from $i\in\mathbb I_N$ stays at points $0$ and $1$ until time $t$. 
Then, we can show the following estimate (\cref{lem:local_time_estimate}) for this local time, which immediately completes the proof of \cref{lem:correlation_estimate} since we have a uniform bound 
$$
\sup_{N\in\mathbb N}\sup_{0\le t\le T}\| \nabla^+_N\rho^N_\cdot(t)\|_{\ell^2(\mathbb T_N)}<C
$$
with some $C=C(T)>0$.

\begin{lemma}
\label{lem:local_time_estimate}
Assume $N\ge 4$. 
Then, there exists some $C=C(D,\mathsf{a} )>0$ such that 
\begin{equation*}
\sup_{0\le t\le T} \max_{i\in\mathbb I_N}
\mathcal T^{(1)}_N(t,i)
\le CT/N.  
\end{equation*}
\end{lemma}
\begin{proof}
Let us consider a function $f:\mathbb I_N\to\mathbb R$ with the following form:
\begin{equation*}
f(i)= -(i-i_0)^2
\end{equation*}
where $i_0=\lfloor N/2\rfloor -1/2$. 
Then, a simple computation shows that 
\begin{equation*}
\begin{aligned}
\mathscr Lf(i)
&= -4D\mathbf{1}_{i\neq 0,1,\lfloor N/2\rfloor}
- p_ND (1-2i+2i_0)\mathbf{1}_{i=\lfloor N/2\rfloor} \\
&\quad - [2D(1+2i-2i_0)+ 2\mathsf{a}  (v_2+1)(1-2i+2i_0) ]\mathbf{1}_{i=1} 
- 4\mathsf{a} (1+2i-2i_0) \mathbf{1}_{i=0} \\
&= -4D\mathbf{1}_{i\neq 0,1,\lfloor N/2\rfloor}
- 4D(2-\lfloor N/2\rfloor) \mathbf{1}_{i=1} 
+2 \mathsf{b} (\lfloor N/2\rfloor -1) \mathbf{1}_{i=1} 
- 4\mathsf{a} (2-2\lfloor N/2\rfloor) \mathbf{1}_{i=0} 
\end{aligned}
\end{equation*}

where in the second identity we used the definition of $i_0$. 
Moreover, by Dynkin's formula, we know that the process 
\begin{equation*}
\begin{aligned}
f(X^{(1)}_{tN^2}) - f(X^{(1)}_0)
- \int_0^t N^2\mathscr L f(X^{(1)}_{sN^2})ds 
\end{aligned}
\end{equation*}
is mean-zero martingale with respect to the natural filtration. 
Now, take the expectation $\mathbf E_i[\cdot]$ in the definition of the martingale in the last display.  
Then for each $t\ge 0$ we have that 
\begin{equation*}
\begin{aligned}
 \mathbf E_i[(X^{(1)}_{tN^2}-i_0)^2] 
-(i-i_0)^2 
&= \int_0^t 4DN^2\mathbf P_i(X^{(1)}_{sN^2} \neq 0,1,\lfloor N/2\rfloor) ds 
\\
&\quad+ \int_0^t \left( 4D \left(\lfloor N/2\rfloor -2 \right) + 2\mathsf{b}   \left(\lfloor N/2\rfloor -1 \right) \right)N^2 
\mathbf P_i(X^{(1)}_{sN^2}=1) ds \\
&\quad- \int_0^t 4\alpha(2\lfloor N/2\rfloor-2)N^2
\mathbf P_i(X^{(1)}_{sN^2}=0)ds \\
&\le 4DN^2t -C_1N^3\int_0^t \mathbf P_i(X^{(1)}_{sN^2} \in \{0,1\}) ds   
\end{aligned}
\end{equation*}
with some $C_1=C_1(D,\mathsf{a} , \mathsf{b}  )>0$.  
Therefore, we have the bound 
\begin{equation*}
\begin{aligned}
\int_0^t \mathbf P_i(X^{(1)}_{sN^2}\in \{0,1\})ds 
\le \frac{4DN^2t+ 4N^2}{C_1N^3} 
\le C_2(t\vee 1)/N
\end{aligned}
\end{equation*}
with another constant $C_2=C_2(D,\mathsf{a} , \mathsf{b}  ) >0$. 
This completes the proof.  
\end{proof}

\section{Proof of \texorpdfstring{\cref{thm:hydrodynamic_general_main}}{hydrodynamic limit}} 
\label{sec:proof_completion} 
In this section we give the proof of the hydrodynamic limit using the results obtained in the previous section. We start by presenting the explicit expressions for the Dynkin martingales. 

\subsection{Estimate of the martingale term}
From \textnormal{(A5)} in {Assumption \ref{assmp:generator_element_basic},}
a direct computation based on Lemma \ref{lem:lap} and summation by parts, shows that
\begin{equation}
\label{eq:dynkin_volume}
\begin{aligned}
M^{N}_t(G)
&=\langle \pi^{N}_t,G\rangle-\langle \pi^{N}_0,G\rangle
-\int_0^tD\langle \pi^{N}_s,\Delta_N G\rangle ds.
\end{aligned}
\end{equation}
Moreover, from Lemma \cref{lem:qua_var} and \cref{K-estimate}, we get that 
\begin{equation*}
\begin{aligned}
\langle M^{N}(G)\rangle_t
&=\int_0^t 
\dfrac{1}{N^2}\sum_{x \in \mathbb{T}_N} \left( \nabla^+_N  G( \tfrac{x}{N}) \right)^2 \big[ D ( \eta_x- \eta_{x+1})^2 
- L_{x,x+1}(\eta_x\eta_{x+1}) \big] (s)d s \\
&\le \int_0^t 
\dfrac{K}{N^2}\sum_{x \in \mathbb{T}_N} \left( \nabla^+_N  G( \tfrac{x}{N})\right)^2 \big( \eta_x^2 +\eta_{x+1}^2 + 1 \big)(s) ds.
\end{aligned}
\end{equation*}
Hence, we conclude up to now that 
\begin{equation}
\label{eq:martingale_vanishment} 
\lim_{N \to \infty}
\mathbb E_N \Big[ \sup_{0\le t\le T} M^N_t(G)^2 \Big] =0
\end{equation}
for any $G\in C^2(\mathbb T)$, where we used Doob's inequality. 

In the next subsection, we prove that the sequence $\{\langle \pi^N_t ,G\rangle\}_{N\in\mathbb N}$ is tight with respect to the Skorohod topology in ${D([0,T],H^{-m}(\mathbb T))}$.

\subsection{Tightness}
\label{subsec:HDL_tightness}

To prove tightness, we follow the approach of \cite[Chapter 11]{kipnis1999scaling}. To that end, we need to introduce some notation. 
\begin{defin}
For $\delta>0$ and a path $\pi$ in ${D([0,T],H^{-m}(\mathbb T))}$, the uniform modulus of continuity of $\pi$, is defined by
\begin{equation*}
\omega_{\delta}(\pi)=\sup_{\substack{|s-t|<\delta, \\0\leq{s,t}\leq{T}}}\|\pi_{t}-\pi_{s}\|_{-m}.
\end{equation*}
\end{defin}
The first result gives sufficient conditions for a subset to be weakly relatively compact.

\begin{lemma}
\label{lem:tightness_sufficient_condition} 
A subset $A$ of $D([0,T],H^{-m}(\mathbb T))$ is relatively compact for the uniform weak topology if $\sup_{Y\in{A}}\sup_{0\leq{t}\leq{T}}\|\pi_{t}\|_{-m}<\infty$ and $\lim_{\delta\rightarrow{0}}\sup_{\pi\in{A}}\omega_{\delta}(\pi)=0$. 
\end{lemma}
From this lemma, we obtain a criterion for tightness of a sequence of probability measures defined on $D([0,T],H^{-m}(\mathbb T))$.
\begin{lemma}\label{lem:tightness_sufficient_condition_sec}
A sequence $\{P_{N}, N\geq{1}\}$ of probability measures defined on $D([0,T],H^{-m}(\mathbb T))$ is tight if the following two conditions hold:
\begin{itemize}
\item[\textnormal{(i)}] $\lim_{A\rightarrow{\infty}}\limsup_{N\rightarrow{\infty}}P_{N}\big(\sup_{0\leq{t}\leq{T}}\|\pi_{t}\|_{-m}>A\big)=0$, 

\item[\textnormal{(ii)}]   $\lim_{\delta\rightarrow{0}}\limsup_{N\rightarrow{\infty}}P_{N}\big(\omega_{\delta}(\pi)\geq{\varepsilon}\big)=0$ for every $\varepsilon>0$. 
\end{itemize}
\end{lemma}

Now, the main result of this section is the following. 

\begin{lemma}
Let {$m>5/2$.}
Then, the sequence of probability measures $\{Q^{N}_m\}_{N\in\mathbb N}$ is tight in $D([0,T],H^{-m}(\mathbb T))$.  
\end{lemma}
 
In order to show that the sequence $\{Q^{N}_m\}_{N\in\mathbb N}$ is tight, it suffices to show the conditions \textnormal{(i)} and \textnormal{(ii)} in \cref{lem:tightness_sufficient_condition_sec} in our context.

\begin{lemma} 
\label{lem:HDL_tighness_key_estimate}
There exists some $C=C(T)>0$ such that for every $z\in\mathbb Z$, 
\begin{equation*}
{\limsup
_{N\rightarrow{+\infty}}}\,\,
\mathbb{E}_{N}\Big[\sup_{0\leq{t}\leq{T}}|\langle\pi^{N}_{t},h_{z}\rangle|^{2}\Big] 
\le Cz^4 .
\end{equation*}
\end{lemma}
\begin{proof}
The proof of this lemma follows from estimating separately each term in the Dynkin martingale with $G=h_z$ given in \eqref{eq:sobolev_CONS_definition}.  
First, note that a simple computation based on a convex inequality together with \eqref{assump:initial_condition_moment}, shows that
\begin{equation*}
\mathbb{E}_{N}
\big[\langle \pi^{N}_{0},h_{z}\rangle^2 \big]
\leq 1 +  
E_{\mu_N} \bigg[\frac{1}{N}\sum_{x\in\mathbb T_N}\eta_x^2 \bigg]
\leq C
\end{equation*}
for some $C=C(T)>0$. 
Similar computations show that the contribution of the martingale also vanishes, by combining Doob's inequality with \eqref{eq:moment_estimate}, see \eqref{eq:martingale_vanishment}.   
Now, we analyze the integral term:
\begin{equation*}
\begin{split}
&\mathbb E_N \bigg[\bigg(\sup_{0\leq{t}\leq{T}}\int_0^t\frac{1}{N}\sum_{x\in\mathbb{T}}
\eta_x(sN^2)\Delta_N h_z(x/N) ds \bigg)^2\bigg] \\
&\quad\le T \mathbb E_N\bigg[ 
\int_0^T \frac{1}{N} \sum_{x\in\mathbb{T}_N}\eta_x(sN^2)^2 
(\Delta_N h_z(x/N))^2 ds  \bigg] \\
&\quad\le T^2 \| h''_z\|^2_{L^\infty(\mathbb T_N)}  
\sup_{0\le t\le T} 
\mathbb E_N\bigg[ \frac{1}{N}\sum_{x\in\mathbb T_N} \eta_x(N^2t)^2 \bigg] 
\le Cz^4  
\end{split}
\end{equation*}
for some $C=C(T)>0$, where we used the moment estimate \eqref{eq:moment_estimate}.  
Hence we conclude with the desired estimate. 
\end{proof}

\begin{coro}
Assume $m>5/2$. Then, 
\begin{equation*}
\limsup_{N\rightarrow{+\infty}}\mathbb{E}_{N}\Big[\sup_{0\leq{t}\leq{T}}\|\pi_{t}^N\|_{-m}^{2}\Big] < \infty 
\end{equation*}
and 
\begin{equation*} 
\lim_{n\rightarrow{+\infty}}  \limsup_{N\rightarrow{+\infty}}\mathbb{E}_{N}\Big[\sup_{0\leq{t}\leq{T}}\sum_{|z|\geq{n}}(\langle \pi_{t}^N,h_{z}\rangle)^{2}\gamma_{z}^{-m}\Big]=0.
\end{equation*}
\end{coro}
\begin{proof}
First, let us show the first item. 
Since
\begin{equation*}
\limsup_{N\rightarrow{+\infty}}\mathbb{E}_{N}\Big[\sup_{0\leq{t}\leq{T}}\|\pi^N_{t}\|_{-m}^{2}\Big]
\le \limsup_{N\rightarrow{+\infty}}
\sum_{z\in{\mathbb{Z}}}\gamma_{z}^{-m} \mathbb{E}_{N}\Big[\sup_{0\leq{t}\leq{T}}\langle \pi^N_{t},h_{z}\rangle ^{2}\Big] ,
\end{equation*}
and from \cref{lem:HDL_tighness_key_estimate} the last display is bounded as long as $m>5/2$. 
The second assertion can be shown analogously. 
\end{proof}

Note that condition \textnormal{(i)} in \cref{lem:tightness_sufficient_condition_sec} holds as a consequence of the first assertion of the previous corollary. It remains now to prove the condition~\textnormal{(ii)}, which follows from the next result.

\begin{lemma}
For every $n\in{\mathbb{N}}$ and every $\varepsilon>{0}$,
\begin{equation*}
\lim_{\delta\rightarrow{0}}
\limsup_{N\rightarrow{+\infty}}
\mathbb{P}_N
\bigg(\sup_{\substack{|s-t|<\delta, \\0\leq{s,t}\leq{T}}} \,
\sum_{|z|\leq{n}}(\langle \pi^N_{t}-\pi^N_{s},h_{z}\rangle)^{2}\gamma_{z}^{-m}>\varepsilon\bigg)=0 . 
\end{equation*}
\end{lemma}
\begin{proof}
The lemma follows from showing that
\begin{equation*}
\lim_{\delta\rightarrow{0}}
\limsup_{N\rightarrow{+\infty}}
\mathbb{P}_N \bigg(\sup_{\substack{|s-t|<\delta,\\0\leq{s,t}\leq{T}}}\, 
(\langle \pi^N_{t}-\pi^N_{s},h_{z}\rangle)^{2}>\varepsilon\bigg)=0
\end{equation*}
for every $z\in{\mathbb{Z}}$ and $\varepsilon>0$.
Recall \eqref{eq:dynkin_volume}. 
Note that we need to show that the compensator term and the martingale term vanish. Let us start with the compensator term. Note that, from the Cauchy-Schwarz inequality for the time integral:
\begin{equation*}
\begin{split}
&\mathbb{P}_N \bigg(\sup_{\substack{|s-t|<\delta,\\0\leq{s,t}\leq{T}}}\, 
\int_s^t\frac{D}{N}\sum_{x\in\mathbb T_N } \Big| \Delta_N h_z(\tfrac xN)\eta_x(sN^2)\Big|ds>\varepsilon\bigg)\\
& \le \frac{1}{\varepsilon^2}\mathbb{E}_N \bigg[\sup_{\substack{|s-t|<\delta,\\0\leq{s,t}\leq{T}}}\, 
\Big(\int_s^t\frac{D}{N}\sum_{x\in\mathbb T_N} \Big| \Delta_N h_z(\tfrac xN)\eta_x(sN^2)\Big| ds\Big)^2\bigg] 
\\
&\le \frac{C}{\varepsilon^2}
\mathbb{E}_N \bigg[\sup_{\substack{|s-t|<\delta,\\0\leq{s,t}\leq{T}}}\, 
(t-s)\int_s^t\Big(\frac{D}{N}\sum_{x\in\mathbb T_N} \Big| \Delta_N
h_z(\tfrac xN)\eta_x(sN^2)\Big| \Big)^2ds\bigg]
\end{split}
\end{equation*}
with some $C>0$. 
Apply again the Cauchy-Schwarz inequality to the summation to bound the last term in the last display by
\begin{equation*}
\frac{C\delta }{\varepsilon_n^2 N^2}
\mathbb{E}_N \bigg[ \int_0^T
\sum_{x\in\mathbb T_N} 
(\Delta_Nh_z(\tfrac xN))^2
\sum_{y\in\mathbb T_N}\eta_y(sN^2)^2 ds\bigg] ,
\end{equation*}
which vanishes as $\delta\to0$, from \eqref{eq:moment_estimate}. 
Hence it remains to show the assertion for the martingale part.
This follows from the next claim: for any function $G\in{C^\infty(\mathbb{T})}$ and  for every $\varepsilon>0$, 
\begin{equation*}
\lim_{\delta\rightarrow{0}}
\limsup_{N\rightarrow{+\infty}}
\mathbb{P}_N \bigg( \sup_{\substack{|s-t|<\delta,\\0\leq{s,t}\leq{T}}} \,
|M_{t}^N(G)-M_{s}^N(G)|>\varepsilon \bigg) =0.
\end{equation*} 
This claim follows from the following estimate: 
\begin{equation*}
\begin{aligned}
\mathbb{P}_N \bigg( \sup_{\substack{|s-t|<\delta,\\0\leq{s,t}\leq{T}}} \,
|M_{t}^N(G)-M_{s}^N(G)|>\varepsilon \bigg)
&\le 2\mathbb{P}_N \bigg( \sup_{t\in[0,T]} \,
|M_{t}^N(G)|>\frac{\varepsilon}{2} \bigg)\\
&\le \frac{8}{\varepsilon^2}\mathbb{E}_N \bigg[\sup_{t,\in[0,T]} \,
|M_{t}^N(G)|^2\bigg]. 
\end{aligned}
\end{equation*}
However, from \eqref{eq:martingale_vanishment}, the utmost right-hand side of the last display vanishes as $N\to+\infty$, and thus we complete the proof.   
\end{proof}

\subsection{Absolute continuity}
As a consequence of the previous subsection, we know that the sequence $\{Q^N_m\}_{N\in\mathbb N}$ has a limit point $Q_m$, by taking a subsequence if necessary.  
Here, we can show that the limit points $Q$ are concentrated on possibly signed measures which are absolutely continuous with respect to the Lebesgue measure. 
Indeed, it easy to see that, by the second-moment bound \eqref{eq:moment_estimate}, for any limit $\pi_\cdot\in D([0,T],H^{-k}(\mathbb T))$, the functional $\pi^N_t$, for each $t\in[0,T]$, can be regarded as a signed measure by extending the domain of test functions.  
Moreover, again from \eqref{eq:moment_estimate}, we can show that Fourier series of the measure $\pi_t(du)$ is in $\ell^2(\mathbb Z)$ a.s. for a.e. $t \in [0,T]$.  
Therefore, following the argument done in \cite[Section 4]{suzuki1993hydrodynamic}, see also \cite[Lemma 3.12]{gonccalves2025stochastic}, we can show the desired assertion, that is,  
\begin{equation*}
Q_m\big( \pi :
\pi_t(u)=\rho(t,u)du \, \text{ a.e.}\, t\big) = 1 . 
\end{equation*}

\subsection{Uniqueness of the weak solution}
Up to now, we know that every limit points of the sequence $\{Q^N_m\}_{N\in\mathbb N}$ are concentrated on absolute continuous trajectories satisfying the weak form of the heat equation \eqref{eq:HDL_heat_equation}: 
\begin{equation*}
Q_m \bigg(\pi : 
\langle \pi_t,G\rangle 
= \langle \pi_0,G\rangle 
+ \int_0^t D\langle\pi_s,G\rangle ds 
\bigg)
= 1 . 
\end{equation*}
Moreover, it is not hard to show that the density $\rho$ is in the space $L^2([0,T]\times \mathbb T)$. 
Under the condition $\|\rho\|_{L^2([0,T]\times \mathbb T)}<+\infty$, 
the weak solution of the heat equation is unique, see~\cite[Section A.2.4]{kipnis1999scaling}.
This allows us to take the full sequence and thus we conclude the desired convergence of $\{Q^N_m\}_{N\in\mathbb N}$.

\appendix
\section{Construction of the dynamics with redistribution interaction}
\label{sec:construction_dynamics}
We comment here on the construction of the dynamics. 
Recall that under Assumption \textnormal{(A4)}, $\nu_\rho$ denotes a product measure whose common marginal is an natural exponential family: $\nu_\rho(d\eta_x)=(1/ Z_{\lambda})e^{\lambda \eta_x} d\nu^o(\eta_x)$ for each $x\in\mathbb T_N$ with some Stieltjes measure $d\nu^o$, and the chemical potential $\lambda=\lambda(\rho)$ is chosen in such a way that $E_{\nu_\rho}[\eta_0]=\rho$. 
Let $p:S\to\mathbb R$ be the probability density function of the one-site marginal distribution of $\nu_\rho$. 
Then the linear operator $L^{\mathrm{Red}}$ defined in \eqref{eq:generator_redistribution} is written with conditional expectation: for each $f\in C(\mathscr X_N)$
\begin{equation}
\label{redistrib-gen}
L^{\mathrm{Red}} f(\eta)
= \frac{1}{2} 
\sum_{x,y \in\mathbb T_N, \, |x-y|=1}
\big( 
 E[f|\eta_x+\eta_y]
 - f(\eta) \big) .
\end{equation}
Here, the conditional expectation is taken with respect to the measure on $S\times S$ whose common marginal is given as that of $\nu_\rho$, so that the law of $\eta_x$ and $\eta_y$ is $\nu_\rho$ for both of them. 
We will denote by $\mathcal D(\mathcal L)$ and $\mathcal R(\mathcal L)$ be the domain and range of a linear operator $\mathcal L$. 
Now, let us recall from \cite{liggett2010continuous} the definition of the probability generator on given in general a state-space $\Omega$, which is locally compact. 

\begin{defin}
\label{def:generator_general}
A probability generator is a linear operator $\mathcal L$ on $C(\Omega)$ satisfying the following properties: 
\begin{itemize}
\item[\textnormal{(a)}] The domain $\mathcal D(\mathcal L)$ is dense in $C_0(\Omega)$. 
\item[\textnormal{(b)}] If $f\in \mathcal D(\mathcal L)$, $\lambda\ge0$ and $g=f-\lambda \mathcal Lf$, then,
\begin{equation*}
\inf_{\eta\in \Omega} f(\eta) \ge \inf_{\eta\in \Omega} g(\eta).  
\end{equation*}
\item[\textnormal{(c)}] $\mathcal R(\lambda-\mathcal L)$ is dense in $C_0(\Omega)$ for all sufficiently small $\lambda>0$. 
\item[\textnormal{(d)}] If $\Omega$ is compact, $1\in \mathcal D(\mathcal L)$ and $\mathcal L1=0$. On the other hand, if $\Omega$ is not compact, for small $\lambda >0$ there exists a sequence $\{f_n\}_{n\in\mathbb N}\subset  \mathcal D(\mathcal L)$ such that $g_n=f_n-\lambda \mathcal Lf_n$ satisfies $\sup_{n}\|g_n\|<+\infty$, and both $f_n$ and $g_n$ converge to 1 pointwise.   \end{itemize}
\end{defin}

Now we check that the operator given in \eqref{redistrib-gen} is a probability generator according to \cref{def:generator_general}.  
Since \eqref{eq:generator_redistribution} is a bounded operator on $C(\mathscr X_N)$, by \cite[Proposition 3.22]{liggett2010continuous}, the condition (c) above immediately follows. 
Let us check (b) and (d). 
To see (d), it is enough to take 
\begin{equation*}
f_n(\eta) 
= \min \Big\{ 1 , \frac{n}{\|\eta\|}\Big\}. 
\end{equation*}
Then, we have $f_n,g_n\to 1$ pointwise and $\|g_n\|$ is uniformly bounded since 
\begin{equation*}
\|g_n\| \le (1+2\lambda N)\|f_n\| .
\end{equation*}
Above we used the fact that our generator is in the form given in \eqref{redistrib-gen}.  
To show (b), let us assume $\mathscr X_N$ is compact.
Then, note that $\inf_{\eta\in \mathscr X_N}f(\eta) \le 0$ due to the fact that $f\in C(\mathscr X_N)$ is decaying at infinity.
If $\inf_{\eta}f(\eta)=0$, since 
\begin{equation*}
g(\eta)
= f(\eta)- \lambda L^{\mathrm{Red}} f(\eta)
= (1+\lambda)f(\eta) 
- \lambda \sum_{x\in\mathbb T_N}
E[f|\eta_{x}+\eta_{x+1}]
\le (1+\lambda)f(\eta), 
\end{equation*}
we have $\inf_\eta g(\eta)\le 0$. 
On the other hand, if $\inf_\eta f(\eta)<0$, the infimum is attained at some point $\eta^*\in \mathscr X_N$. 
Then, 
\begin{equation*}
\begin{aligned}
g(\eta^*)
&= f(\eta^*) - \lambda L^{\mathrm{Red}} 
f(\eta^*)\\
&= f(\eta^*) 
- \lambda \sum_{x\in\mathbb T_N} 
\big( E[f|\eta_x+\eta_{x+1}](\eta^*) - f(\eta^*)\big)
\le f(\eta^*)
= \inf_{\eta\in\mathscr X_N} f(\eta). 
\end{aligned}
\end{equation*}
Now, if $\mathscr X_N$ is compact, the infimum of $f$ is attained at some point, so that the argument in the last display immediately verifies the condition (b). 
Hence we could check all items of \cref{def:generator_general} and thus the desired dynamics is constructed, according to \cite[Theorem 3.24]{liggett2010continuous}.

\section{Verification of the martingale properties}
\label{sec:degree-preserving_verification} 
Here we show that all the microscopic models that we listed as examples satisfy the assumptions \textnormal{(A6.1)} and \textnormal{(A6.2)} on the degree preservation listed in Assumption \ref{assumption:second_order_polymonial}, especially the martingale property in the assumptions. 
Recall that we introduced the processes $M_f=(M_f(t):t\ge0 )$ and $N_f=(N_f(t):t\ge0)$ by \eqref{eq:polynomial_martingale} and \eqref{eq:polynomial_martingale_qv}, respectively. 
Take any polynomial $f\in \mathcal P_2$. 
Hereinafter let us focus on the proof of $M_f$ being a martingale, since the assertion for $N_f$ can be shown analogously. 
To show that $M_f$ is a martingale, we need to take a sequence $\{f_n\}\subset C(\mathscr X_N)$. 
Let $\{\mathscr F_t:t\ge 0\}$ be the natural filtration of the process $\{\eta(t):t\ge 0\}$.  
Note that by Dynkin's martingale formula, we know that the process $M_{f_n}$ is a martingale.
Thus, for any $A\in \mathscr F_s$ and for any $0 \leq s \leq t$,  
\begin{equation*}
\mathbb E_N[M_{f_n}(t)\mathbf{1}_A]
= \mathbb E_N[M_{f_n}(s)\mathbf{1}_A]. 
\end{equation*}
Therefore, to show that $M_f$ is a martingale, it is enough to seek for a sequence $\{f_n\}_{n\in\mathbb N}$ such that $M_{f_n}\to M_f$ almost surely, and $\{M_{f_n}\}_{n\in\mathbb N}$ is uniformly integrable. 
To this end, notice that we may only show the assertion for 
\begin{equation}
\label{eq:martingale_function_domination}
f(\eta)
= \sum_{x\in\mathbb T_N} \eta_x^2. 
\end{equation}
In particular, we can show some uniform bound on (the expectation of) this function with the help of Gronwall's inequality. 
For the other cases, approximation by elements in $C(\mathscr X_N)$ can be constructed analogously to the previous case, whereas to derive the uniform bound, we use the fact that the function $f$ given above dominates any function in the set $\mathcal P_2$ up to a constant shift and multiplication.    
In what follows, we take $f$ to be the one in \eqref{eq:martingale_function_domination}, and, to make use of some model-wise properties, we split the proof into three cases in the following way. 

\subsection{Case I: Non-negative state-space}
First, let us consider the case where the state-space is non-negative: $S\subset [0,+\infty)$. 
An idea for this case is to use the conservation law. 
Let $f_n\in C(\mathscr X_N)$ be defined by 
\begin{equation*}
f_n(\eta)
= f(\eta)
\Theta_n\Big(\sum_{x\in\mathbb T_N}\eta_x\Big)
\end{equation*}
where $\Theta_n:[0,+\infty)\to [0,1]$ is a bounded smooth function such that $\Theta_n(r) =1$ if $r\le n$ whereas $\Theta_n(r)=0$ if $r\ge n+1$. 
Now, note that since $L:\mathcal P_2\to \mathcal P_2$, we have the bound 
\begin{equation*}
Lf_n(\eta)
\le K_1 f_n(\eta) + K_2
\end{equation*}
for some $K_1,K_2\ge 0$ which are independent of $n$, where we used the conservation law when computing the action of the generator.  
By Dynkin's martingale formula, we have that 
\begin{equation*}
\mathbb E_N[f_n(\eta(t))]
\le \mathbb E_N[f_n(\eta(0))]
+ \int_0^t \big( K_1 \mathbb E_N[f_n(\eta(s))] + K_{2}\big) ds ,
\end{equation*}
which shows that $\mathbb E_N[f_n(\eta(t))]$ is bounded uniformly in $n$ since from \eqref{assump:initial_condition_moment} we know that this bound holds at time $t=0$.   
Hence, the proof ends since the function $f$ defined in \eqref{eq:martingale_function_domination}is in $L^1(\mathbb P_N)$ and it dominates all functions in $\mathcal P_2$.

\subsection{Case II: Continuous state-space} 
Next, let us consider the case $S=\mathbb R$ and the process $(\eta(t):t\ge 0)$ has continuous trajectories, i.e. when it is given as an interacting diffusion. 
Let 
\begin{equation*}
\sigma_n
=\inf\bigg\{ t\ge 0 ; \sum_{x\in\mathbb T_N} \eta_x(t)^2 \ge n \bigg\}.
\end{equation*}
Moreover, let 
\begin{equation*}
f_n(\eta) 
= f(\eta) 
\Theta_{n+1}\Big( \sum_{x\in\mathbb T_N} \eta_x^2\Big) 
\end{equation*}
where recall that the function $\Theta_n$ is the same as in Case I. 
Here note that we took the cutoff at the value $n+1$, in order for the identity $Lf(\cdot)= Lf_n(\cdot)$ to be true until the random time $\sigma_n$. 
Then, $f_n\in C(\mathscr X_N)$ for each $n\in\mathbb N$.  
Moreover, we know that $M_{f_n}$ converges almost surely to $M_f$, and again by Dynkin's martingale formula, the process 
\begin{equation*}
M^{\sigma_n}_{f_n}(t)
\coloneqq f_n(\eta(t\wedge \sigma_n))
- f_n(\eta(0))
-\int_0^{t\wedge \sigma_n} 
L f_n(\eta(s)) ds 
\end{equation*}
is a martingale, since $\sigma_n$ is a stopping time.  
Then, we can conduct a similar argument as in Case I to show that, with the help of Gronwall's inequality, $\mathbb E_N[f_n(\eta(t\wedge \sigma_n))]$ is uniformly bounded in $n$,  so that we have $\mathbb E_N[f(\eta(t))] 
= \lim_{n\to\infty} 
\mathbb E_N[f_n(\eta ( t \wedge \sigma_n))]
<+\infty$. 
Hence, the function $f$ defined in \eqref{eq:martingale_function_domination} is integrable, and we complete the proof for this case.

\subsection{Case III: Redistribution-type interaction} 
Finally, let us focus on the case for models with redistribution-type interaction. 
Define for any $m,n\in\mathbb N$ satisfying $m>n$, 
\begin{equation*}
f^m_n(\eta)
= \begin{cases}
    \begin{aligned}
    & f(\eta) && \text{ if } f(\eta) \le n, \\
    & n\frac{m-f(\eta)}{m-n} && \text{ if } n \le f(\eta) \le m, \\
    & 0 && \text{ if } f(\eta) \ge m .
    \end{aligned}
\end{cases}
\end{equation*}
Here, note that $f^m_n(\eta)\uparrow f(\eta)\wedge n \eqqcolon f_n(\eta)$ as $m\to\infty$ almost surely and $f_n$ is not included in the space $C(\mathscr X_N)$, though $f^m_n\in C(\mathscr X_N)$ for any fixed $m$.   
Now, we claim that the bound holds: 
\begin{equation}
\label{eq:redistirbution_dominating_function_bound}  
|Lf^m_n(\eta)| 
\le K_1f_n(\eta)+ K_2
\end{equation}
with some $K_1,K_2\ge 0$ which do not depend on $m,n$.  
Indeed, note that 
\begin{equation*}
\begin{aligned}
Lf^m_n(\eta)
&=(1/2) \sum_{x,y\in\mathbb T_N, \, |x-y|=1}
\int_S 
q_\alpha(\eta_x,\eta_y) \big( f^m_n(\eta -\alpha \delta_x + \alpha \delta_y)-f^m_n(\eta) \big)d\alpha\\
&\le (1/2) \sum_{x,y\in\mathbb T_N, \, |x-y|=1}
\int_S 
q_\alpha(\eta_x,\eta_y) \big( f(\eta -\alpha \delta_x + \alpha \delta_y)-f^m_n(\eta) \big)d\alpha\\
&\le Lf(\eta) + N\big(f(\eta) - f^m_n(\eta)\big)
\end{aligned}
\end{equation*}
where in the last estimate, we added and subtracted $f(\eta)$ inside the integral and used the normalization condition of the rate $q_\alpha$. 
Therefore, since the function $f^m_n$ is non-negative, and the action $L$ on $f$ is again degree-two, we have the bound 
\begin{equation*}
Lf^m_n(\eta)
\le K_1f(\eta)+ K_2
\end{equation*}
with some $K_1,K_2\ge 0$. 
Note that when $f(\eta)\leq n$ we have $f_n(\eta)=f(\eta)$, and then the last inequality reads $
Lf^m_n(\eta)
\le K_1f_n(\eta)+ K_2$. 
On the other hand, we note that the non-negativity of $f^m_n$ and the fact that $f^m_n\le n$ yields the bound 
\begin{equation*}
Lf^m_n(\eta)
= (1/2) \sum_{x,y\in\mathbb T_N,\, |x-y|=1}  
\int_S
q_\alpha(\eta_x,\eta_y)\big( 
f^m_n(\eta- \alpha \delta_x + \alpha \delta_y) - f^m_n(\eta) \big)
d\alpha
\le \widetilde K_1 n
\end{equation*}
with another $\widetilde K_1\ge 0$. Again we note that when $f(\eta)\geq n$, we have $f_n(\eta)=n$ and therefore, the last bound reads $
Lf^m_n(\eta)
\le \widetilde K_1 f_n(\eta)$. These two observations 
 give the upper bound of \eqref{eq:redistirbution_dominating_function_bound}.
Analogously, we can bound $Lf^m_n(\eta)$ from below by $-(K_1 f_n(\eta) + K_2)$ and thus we obtain the claim \eqref{eq:redistirbution_dominating_function_bound}. 
Now, by Dynkin's martingale formula, we have that 
\begin{equation*}
\begin{aligned}
\mathbb E_N[f^m_n(\eta(t))]
&= \mathbb E_N[f^m_n (\eta(0))]
+ \mathbb E_N\bigg[ \int_0^t Lf^m_n(\eta(s))ds \bigg] \\
&\le \mathbb E_N[f^m_n(\eta(0))] 
+ \int_0^t \big( K_1 \mathbb E_N[f_n(\eta(s))] 
+ K_2 \big) ds.
\end{aligned}
\end{equation*}
Hence, by taking $m\to \infty$ by the monotone convergence theorem, and then applying Gronwall's inequality, we conclude with the uniform bound $\sup_{t \in [0,T]}\sup_{n\in\mathbb N} \mathbb E_N[f_n(\eta(t))] < +\infty$. 
Now, take the limit $n\to\infty$, again by the monotone convergence theorem to deduce the desired bound for the function $f$ and we complete the proof for redistribution models.

\subsection*{Acknowledgments}
P.G. and K.H. thank the hospitality of Universidade do Minho for their research visit during the period of September 2025 when part of this work was developed.   
K.H. is supported by KAKENHI 25K23337.  
P.G. and C.F. thank the hospitality of the University of Tokyo and the University of Osaka for their research stay during the period of October 2025 when part of this work was developed.
C.F. acknowledges support from the FAR UniMoRe project CUP-E93C25002460005.
P.G. thanks Fundação para a Ciência e Tecnologia
FCT/Portugal for financial support through the projects UIDB/04459/2020 and UIDP/04459/2020 and ERC-FCT.
M.S. is supported by KAKENHI 24K21515.

\subsection*{Data Availability}
No datasets were generated or analyzed during the current study. 

\subsection*{Conflict of Interests}
The authors declare that they have no known competing financial interests or personal relationships that could have appeared to influence the work reported in this paper.

\bibliographystyle{abbrv}
\bibliography{ref}

\end{document}